\documentclass[reqno]{amsart}
\usepackage[german,english]{babel}
\usepackage[utf8]{inputenc}
\usepackage{amsmath,amssymb,amsthm,amsfonts}
\usepackage{emptypage}

\usepackage[toc,page]{appendix}
\usepackage[dvipsnames]{xcolor}
\numberwithin{equation}{section}
\numberwithin{table}{section}
\numberwithin{figure}{section}

\usepackage[shortlabels]{enumitem}
\usepackage{float}
\usepackage{bbm}
\usepackage{hyperref}
\usepackage[notref,notcite,final]{showkeys}

\usepackage{setspace}
\usepackage{array}
\usepackage{tikz-cd}
\usepackage{tabularx}
\usepackage{pdfpages}
\usepackage{mathrsfs}
\usepackage{dutchcal}
\usepackage{mathtools}

\newtheorem{satz}{Theorem}[section]
\newtheorem{lemma}[satz]{Lemma}
\newtheorem{korollar}[satz]{Corollary}
\theoremstyle{definition}
\newtheorem{definition}[satz]{Definition}

\newtheorem{hyp}[satz]{Hypothesis}
\theoremstyle{remark}
\newtheorem{bemerkung}[satz]{Remark}
\theoremstyle{plain}
\newtheorem{proposition}[satz]{Proposition}

\setlist[enumerate]{align=left,itemindent=*,labelsep=12pt}
\newlist{proofenum}{enumerate}{1}
\setlist[proofenum]{label=\upshape(\roman*),itemindent=*,labelsep=-4pt,align=left,leftmargin=0em,itemsep=0em}
\newcommand{\C}{{\mathbb{C}}}

\newcommand{\N}{{\mathbb{N}}}

\newcommand{\R}{{\mathbb{R}}}

\newcommand{\BB}{{\mathscr{B}}}
\newcommand{\CC}{{\mathscr{C}}}

\newcommand{\NN}{{\mathscr{N}}}

\newcommand{\Aa}{{\mathcal{A}}}

\newcommand{\Cc}{{\mathcal{C}}}

\newcommand{\Gg}{{\mathcal{G}}}
\newcommand{\Hh}{{\mathcal{H}}}

\newcommand{\Tt}{{\mathcal{T}}}

\newcommand{\one}{{\mathbbm{1}}}
\newcommand{\ind}[1]{{\mathbbm{1}_{#1}}}

\renewcommand{\a}{\mathfrak{a}}
\renewcommand{\b}{\mathfrak{b}}

\renewcommand{\d}{{\mathrm{d}}}

\newcommand{\id}{\mathop{\textrm{\upshape{id}}}}

\newcommand{\BUC}{\mathop{\textrm{\upshape{BUC}}}}
\renewcommand{\epsilon}{\varepsilon}

\newcommand{\symb}{\mathop{\text{\upshape{symb}}}}

\newcommand{\dm}{d}

\renewcommand{\Re}{{\rm Re\, }}

\newcommand{\dist}{{\rm dist}}
\newcommand{\supp}{{\rm supp}}

\renewcommand{\epsilon}{\varepsilon}

\renewcommand{\a}{{\mathfrak{a}}}
\newcommand{\tr}{\mathop{\textrm{\upshape{tr}}}}
\newcommand{\<}{\left\langle}
\renewcommand{\>}{\right\rangle}
\renewcommand{\phi}{\varphi}
\newcommand{\eps}{\varepsilon}

\newcommand{\trinom}[3]{\begin{pmatrix} {#1}\\{#2}\\{#3} \end{pmatrix}}

\newcommand{\la}{\langle}

\newcommand{\ra}{\rangle}
\newcommand{\dx}{{\mathrm{d}}}
\newcommand{\Div}{\mathop{\textrm{\upshape{div}}}}
\newcommand{\uu}{\mathcal{u}}
\newcommand{\vv}{\mathcal{v}}

\newcommand{\ff}{\mathcal{f}}
\newcommand{\ee}{\mathcal{e}}

\renewcommand{\a}{{\mathfrak{a}}}
\renewcommand{\b}{{\mathfrak{b}}}

\newcommand{\pbo}{{\pagebreak[3]}}
\usepackage{graphicx} % Required for inserting images

\allowdisplaybreaks
\begin{document}
\title[Elliptic 4th-order operators with Wentzell boundary conditions]{Elliptic fourth-order operators with Wentzell boundary conditions on Lipschitz domains}
\author{David Plo\ss}
\address{D.\ Plo\ss, Karlsruhe Institute of Technology, Department of Mathematics, Englerstraße 2, 76131 Karlsruhe, Germany}
\email{david.ploss@kit.edu}
\date{March 2024}
\thanks{Funded by the Deutsche Forschungsgemeinschaft (DFG, German Research Foundation) –
Project-ID 258734477 – SFB 1173.}
\thanks{The author thanks Robert Denk and Markus Kunze for helpful comments and discussions.}
\begin{abstract}
For bounded domains $\Omega$ with Lipschitz boundary $\Gamma$, we investigate boundary value problems for elliptic operators with variable coefficients of fourth order subject to Wentzell (or dynamic) boundary conditions. Using form methods, we begin by showing general results for an even wider class of operators of type $$\Aa=\begin{pmatrix} B^*B & 0 \\ -\NN_\b B & \gamma \end{pmatrix},$$ where $B$ is associated to a quadratic form $\b$ and $\NN_b$ an abstractly defined co-normal Neumann trace. Even in this general setting, we prove generation of an analytic semigroup on the product space $\Hh:=L^2(\Omega) \times L^2(\Gamma)$. Using recent results concerning weak co-normal traces, we apply our abstract theory to the elliptic fourth-order case and are able to fully characterize the domain in terms of Sobolev regularity for operators in divergence form $B=-\Div Q \nabla$ with $Q \in C^{1,1}(\overline\Omega,\R^{d\times d})$ , also obtaining Hölder-regularity of solutions. Finally, we also discuss asymptotic behavior and (eventual) positivity. 
\end{abstract}


\date{\today}

\maketitle

\section{Introduction}
 Wentzell, or dynamic boundary conditions, appear in a multitude of physical applications and pose a mathematically challenging problem. Given a bounded domain $\Omega$ with boundary $\Gamma$, they model the interchange of free energy of a physical system between $\Omega$ and $\Gamma$. The main issue with modeling this interchange is that the energy flux is  represented by an integral over the domain, which cannot ``see'' the boundary as it is a set of Lebesgue measure zero. This is usually resolved by considering functions in a product space, e.g., $\Hh:=L^2(\Omega)\times L^2(\Gamma)$, and a related operator $\Aa$ for which the action in the interior of the domain and on the boundary is decoupled (cf.\ \cite{AE96, AMPR03,Eng03}). The connection between interior and boundary is then encoded by a coupling condition in the definition of the domain of $\Aa$. 

A detailed discussion of a physical interpretation of these boundary conditions in comparison to classical ones can be found in~\cite{Gol06} for the two most famous PDE, the heat and the wave equation. Other instances where these boundary conditions occur are the Stefan problem with surface tension (see~\cite[Section~1]{Escher-Pruess-Simonett03}), and climate models including coupling between the deep ocean and the surface (see~\cite[Section~2]{DT08}) where they incorporate the external energy transported into the ocean by the sun. Furthermore, they are used in the Cahn--Hilliard equation describing spinodal decomposition of binary polymer mixtures (see~\cite[Section~1]{Racke-Zheng03}) in order to model effects close to the boundary, e.g., that one of the agents is more attracted to the boundary than the other, which might lead to further separation effects. Contrary to the first examples whose leading part is given by (variations of) the Laplacian, the Cahn--Hilliard equation is based on the Bi-Laplacian, an operator of order 4 which fits into the setting of the present work.

We are going to consider a general class of operators. As prototype and main application, we study the following system of fourth order:
\begin{alignat}{4}
  \partial_t u   +B(\alpha  B)u & =0 && \textnormal{ in }   (0,\infty)\times \Omega,\label{10-1}\\
  \tr B(\alpha  B)u -\beta  \partial^Q_\nu (\alpha  B) u -  \beta\delta \tr (\alpha B) u -\gamma \tr u   & = 0
   && \textnormal{ on }  (0,\infty)\times \Gamma,\label{10-2}\\
  \partial^Q_\nu u+ \delta \tr u   & = 0  && \textnormal{ on }   (0,\infty)\times \Gamma,\label{10-3}\\
  u |_{t=0} & = u_0   && \textnormal{ in } \Omega.\label{10-4}
\end{alignat}
Here $B$ is given by $B=-\Div Q \nabla$, where $Q \in C^{1,1}(\overline\Omega, \R^{\dm \times \dm})$ is uniformly positive definite and $\Omega$ is a Lipschitz domain. Let $\partial_\nu^Q$ denote its corresponding co-normal derivative given by $\<\nu,Q \nabla u\>$ (cf.\ \eqref{Def:WeakNeuA} below), and assume $\alpha,\beta,\gamma, \delta$ to be bounded, real-valued functions. The precise smoothness assumption of the coefficients will be specified later on (cf.\ Hypotheses~\ref{hyp1.1} and \ref{HypQ}). 
In~\eqref{10-1}--\eqref{10-4}, it is implicitly assumed that the initial value $u_0$ is sufficiently smooth to have a trace on the boundary and that this trace is used as an initial condition for $u$ on the boundary.

Note that, as Equation~\eqref{10-1} is of fourth order with respect to $x\in\Omega$, we have to impose two boundary conditions. Here, we have chosen the Robin boundary condition~\eqref{10-3} in addition to the Wentzell boundary condition~\eqref{10-2}. 

The main mathematical challenge in tackling Wentzell boundary conditions lies in the fact that the elliptic operator that governs the equation in the interior itself appears in the boundary condition, and the standard condition $B(\alpha B) u \in L^2(\Omega)$ is not sufficient to guarantee existence of the trace. In order to decouple this system and circumvent this issue, we rewrite the Wentzell boundary condition \eqref{10-2} as a dynamic boundary condition using $B(\alpha B)u =-\partial_t u$ from \eqref{10-1}. Then we rename $u$ to $u_1$ and replace the time derivative $\partial_t u_1$ in the boundary condition by the time derivative $\partial_t u_2$ of an independent function $u_2$ that lives on the boundary. Even though $u_2$ is formally independent of $u_1$, we think of $u_2$ as the trace of $u_1$; this condition will be incorporated into the domain of our operator $D(\Aa)$, later. We thus obtain the following decoupled \index{Wentzell boundary conditions!Decoupled system} version of~\eqref{10-1}--\eqref{10-4}:
\begin{alignat}{4}
  \partial_t u_1   +B(\alpha  B)u_1 & =0 && \textnormal{ in }   (0,\infty)\times \Omega,\label{10-5}\\
  \partial_t u_2 +\beta  \partial^Q_\nu (\alpha  B) u_1 +  \beta\delta \tr (\alpha B) u_1 + \gamma  u_2   & = 0
   && \textnormal{ on }  (0,\infty)\times \Gamma,\label{10-6}\\
  \partial_\nu^Q u_1+\delta u_2  & = 0  && \textnormal{ on }   (0,\infty)\times \Gamma,\label{10-7}\\
  u_1 |_{t=0} & = u_{1,0}  && \textnormal{ in } \Omega,\label{10-8}\\
  u_2 |_{t=0} & = u_{2,0} && \textnormal{ on } \Gamma.\label{10-9}
\end{alignat}
Note that, as $u_2$ is independent of $u_1$, we have to impose an additional initial condition for $u_2$. If, however, the initial value $u_0$ in~\eqref{10-4} is smooth enough, we can put $u_{1,0}=u_0$ and $u_{2,0}=u_0|_\Gamma$.
However, this coupling condition for the initial value is not mandatory. The Hilbert space theory established below will allow to prescribe non-continuous initial data, even the extreme case of $u_{1,0}=0$ and $u_{2,0} \in L^2(\Gamma)$ arbitrary is allowed. This is especially useful to model situations where in the beginning the entire energy only lives on the boundary and slowly dissipates into the interior of the domain over time.
Rewritten as a Cauchy problem, for $\uu=(u_1,u_2) \in \Hh$ we obtain
\begin{alignat}{4}
\partial_t \uu + \Aa \uu& =0 && \textnormal{ for } \uu(t,\cdot) \textnormal{ in } \Hh, \\
\uu|_{t=0} &= \uu_0 && \textnormal{ for } \uu_0 \in \Hh, 
\end{alignat}
where $\Aa$ is given by 
\begin{equation}\label{DefAGen}
\begin{pmatrix}
    B(\alpha B) & 0\\
    \beta (\partial_\nu^Q(\alpha B) + \delta \tr(\alpha B)) & \gamma
\end{pmatrix} 
\end{equation}
on a suitable domain $D(\Aa)$ that incorporates~\eqref{10-7} and the coupling condition $u_2=\tr u_1$. In order to construct a solution for~\eqref{10-5}--\eqref{10-9}, the main idea is to obtain an analytic semigroup (generated by $-\Aa$), whose smoothing effects will allow us to recover the original system with Wentzell boundary conditions. 
Using this decoupling idea and form methods to tackle Wentzell boundary conditions, has proven to be a suitable approach for the second-order case, e.g., the Laplace operator subject to Wentzell boundary conditions. A series of papers starting in 2003 has shown generation results concerning an analytic semigroup for the decoupled system on $L^2(\Omega)\times L^2(\Gamma)$, using the classical Beurling--Deny criteria \cite{AMPR03}. These results were then extended to the $L^p$-scale, and later also to general second-order elliptic operators on Lip\-schitz domains, where also Hölder continuity of the solution was deduced, see \cite{Nit11} and \cite{War13}. Under additional smoothness assumptions also spaces of continuous functions were considered in \cite{AMPR03}; see also \cite{EF05} and \cite{BE19} where generation of an analytic semigroup was shown in an abstract perturbation framework. For higher order elliptic operators the extension procedure to the $L^p$-scale does not work, because the Beurling--Deny criteria are in general not fulfilled (see also Proposition~\ref{nonpos}). Less results are available and they typically rely on being in a smooth setting. For fourth-order equations with sufficiently smooth coefficients in $C^4$-domains, it was shown in \cite[Theorem~2.1]{FGGR08} that the related operator in the product space is essentially self-adjoint. For the Cahn--Hilliard equation, classical well-posedness was shown in \cite[Theorem~5.1]{Racke-Zheng03} in the $L^2$-setting, and in \cite[Theorem~2.1]{Pruess-Racke-Zheng06} in the $L^p$-setting. Again the domain and the coefficients were assumed to be (sufficiently) smooth, and the methods do not carry over to Lipschitz domains.

In \cite{Denk-Kunze-Ploss21}, the Lipschitz-case was solved for the Bi-Laplacian using weak Green's formulae and the theory of quasi-boundary triples \cite[Chapter 8]{BHS20}. In the present paper, however, we choose a more abstract approach which deals with a larger class of systems and does not depend on the theory of boundary triples. It contains the results of \cite{Denk-Kunze-Ploss21} as a special case, also giving a simpler proof employing recent developments on co-normal derivatives in Lipschitz domains for operators with variable coefficients \cite{BGM22}.  

To that end, in Section~\ref{SecGenForm}, we begin by addressing very general forms $\b$ whose associated operators will take the role of $B$ in \eqref{DefAGen}. To proceed, we start by considering the Neumann case $\delta=0$ and define an abstract Neumann trace $\NN_b$ that fits into the form approach and is connected to Green's second formula. Afterwards, in Section~\ref{HoSys}, we investigate the quadratic form $\mathfrak a$ on the product space $\Hh=L^2(\Omega) \times L^2(\Gamma)$ to which the operator $\mathcal A$ is associated. Based on the analysis of that form, we can show that the operator $\mathcal A$ is self-adjoint and $-\Aa$ is the generator of a strongly continuous and analytic semigroup $(\mathcal T(t))_{t\ge 0}$ (Theorem~\ref{GenOP2}). This will also show that the operator $\mathcal{A}$ indeed governs a generalized version of~\eqref{10-5}--\eqref{10-9} with $\delta=0$, given by~\eqref{A10-5''}--\eqref{A10-9''}. We will explain that we can also obtain a solution of the Wentzell system in the original formulation generalizing ~\eqref{10-1}--\eqref{10-3} with initial condition~\eqref{10-4}. If $u_{2,0}$ is not the trace of $u_{1,0}$, there are some subtleties concerning the initial values, see Remark~\ref{r.explain}. 
In Section~\ref{id}, we return to the main application where $B$ is a second-order elliptic operator in divergence form, identifying the associated operator $B$ and its minimal and maximal realization, as well as the normal trace $\NN_b$ (Section~\ref{SecA1}). In Section~\ref{id2}, we finally collect our results for the fourth-order system, which culminate in Corollary~\ref{Cor:A3} where we precisely identify the operator $\mathcal A$ and its domain in terms of Sobolev regularity. After extending our results to the Robin case $\delta>0$, we obtain that the operator $\mathcal{A}$ indeed governs the system precisely as formulated in~\eqref{10-5}--\eqref{10-9}.

In Section~\ref{Prop}, we  briefly discuss higher regularity for smoother domains and coefficients (Section~\ref{Smooth3.5}) before undertaking further investigations of the operator $\Aa$ in the original setting. One of the main results of this section is Theorem~\ref{HoelReg}, which states that for every element $(u_1, u_2)$ of $D(\mathcal A^\infty)$ the function $u_1$ is H\"older continuous and $u_2$ is the trace of $u_1$. As the semigroup $\Tt$ is analytic, it follows that for positive time the solution of~\eqref{10-5}--\eqref{10-9} is H\"older continuous and satisfies the  Wentzell boundary condition in a pointwise sense. Moreover, this result implies regularity of the eigenfunctions of the operator $\Aa$ and is used later on.
In Section~\ref{spec}, we show that the operator $\mathcal{A}$ has compact resolvent and thus a decomposition into a basis consisting of eigenfunctions of $\mathcal{A}$. This allows us to describe the semigroup in terms of the  eigenfunctions and to characterize the asymptotic behavior of the semigroup. In particular, we study its positivity properties: It turns out that the generated semigroup is neither positive nor $L^\infty$-contractive (Proposition~\ref{nonpos}) as the operator does not satisfy the Beurling--Deny criteria. 
However, as shown for the for the semigroup generated by the Bi-Laplacian in \cite{Denk-Kunze-Ploss21}, in the case $\gamma=\delta=0$ our semigroup is again eventually positive in the sense of \cite{DGK16b} and \cite{DG18}
 (Theorem~\ref{EventPos}). We close the article by showing that the same abstract approach can be used to obtain abstract results for higher-order operators, e.g., $(-\Delta)^{4k}$. 
\section{The abstract setting}\label{Abstract}
Aim of this section is to establish a solution theory for Wentzell boundary conditions for higher-order operators which can be represented by nested forms, i.e.\ two quadratic forms where the operator associated to the first one is used to construct the second. We fix the following setting:

Let $\Omega \subset \R^\dm$ be a domain with Lipschitz boundary $\Gamma$.
 We denote the inner products in  $L^2(\Omega)$ and $L^2(\Gamma)$ by
\[
\la f,g \ra_\Omega \coloneqq \int_\Omega f\overline{g}\, \dx  x\quad \mbox{and} \quad \la f, g\ra_\Gamma \coloneqq \int_\Gamma f\overline{g}\, \dx S,
\]
respectively, and write $\|\cdot\|_\Omega$ and $\|\cdot\|_\Gamma$  for the induced norms.
We denote the standard Sobolev spaces by $H^s(\Omega)$ for $s \geq 0$.
By slight abuse of notation, we will also write
\[
\la \nabla u, \nabla v\ra_\Omega\coloneqq \int_\Omega \sum_{j=1}^\dm \partial_j u\overline{\partial_j v}\, \dx x
\]
whenever $u, v\in H^1(\Omega)$. For fractional orders we may either use complex interpolation, or, equivalently, restriction. For negative orders we employ duality.
Additionally for an elliptic operator of second-order $B$, we introduce the space $H^s_B(\Omega)$ for the space of functions $u\in H^s(\Omega)$ such that $Bu$ belongs to $L^2(\Omega)$ (cf.\ Definition~\ref{DefBDist}). We endow $H^s_B(\Omega)$ with the canonical norm
\[ \|u\|_{H^s_B(\Omega)}^2 \coloneqq \|u\|_{H^s(\Omega)}^2 + \|Bu\|_\Omega^2 \quad (u\in H^s_B(\Omega)).\]
The Dirichlet-trace on $C^\infty(\overline \Omega)$, defined by $u \mapsto u_{|\Gamma}$, and its extension to any Sobolev space $H^{s}(\Omega)$ for $s>\frac{1}{2}$ is denoted by $\tr$.

Taking a brief look back to the Bi-Laplacian case (cf.\ \cite{Denk-Kunze-Ploss21}), where the form $\a(u,v)=\<\Delta_N u, \Delta_N v\>$ is considered on the domain $$\{ \uu=(u_1,u_2) \in \Hh ~|~ u_1 \in D(\Delta_N), u_2=\tr u_1\},$$ we recall that its associated operator is given by 
\[\Aa=\begin{pmatrix}
    \Delta^2 & 0 \\ -\partial_\nu \Delta  & 0
\end{pmatrix}. \] In order to tackle the general system, in the form $\a$, we are going to replace the Neumann Laplacian $\Delta_N$ by a more general operator $B_N$. To that end, we introduce a second form that somehow operates on a ``lower level''. More precisely, $\Delta_N$ is naturally associated to the form $\b(u,v)=\<\nabla u,\nabla v\>_\Omega$ on $L^2(\Omega)$ with form domain $D(\b)=H^1(\Omega),$ so we may generalize this form. In order to distinguish between $\a$ and $\b$ terminologically, we will call $\a$ the primary form and $\b$ the subsidiary form.\index{Forms!subsidiary} \index{Forms!primary}

At first, we will establish our theory for quite general subsidiary quadratic forms $\b$ whose associated operators are not necessarily differential operators in divergence form or even of second order. 

\subsection{Abstract realizations of the lower-order operator}\label{SecGenForm}
 Recall, that if a form $\b: D(\b) \times D(\b) \rightarrow \C ~(D(\b)\subset H)$ is densely-defined, semi-bounded by $\lambda \in \R$, closed, and continuous in the sense of \cite[Chapter~1]{Ouh09}, its associated operator $A$ satisfies that $\lambda-A$ generates an analytic contraction semigroup on $\Hh$. We call such forms \emph{generating}. 
 Note that, in this terminology, $\b$ is semi-bounded by $\lambda$ if the shifted form $\b_\lambda(u,v)=\b(u,v)+\lambda\<u,v\>_H$ is accretive. Furthermore, $A$ will be self-adjoint if $\a$ is also symmetric.

\begin{definition}  \label{DefSubForm}\index{Forms!admissible}
 Consider the Hilbert space $H=L^2(\Omega)$. We call a form $\b: D(\b) \times D(\b) \rightarrow \C$ \textit{admissible}, if it is a generating, symmetric form on $H$ such that for some $\rho \in (0,1)$ \begin{equation}\label{TestDense}
    C_c^\infty(\Omega) \subseteq (D(\b),\|\cdot\|_\b) \subseteq H^{\frac{1}{2}+\rho}(\Omega)
\end{equation} holds, where the latter embedding is continuous and dense. 
\begin{bemerkung}
    The continuous embedding into the space $H^{1/2+\rho}(\Omega)$ is assumed to ensure existence of the Dirichlet trace. The space $H^{1/2+\rho}(\Omega)$ can be replaced by any space on which the Dirichlet trace exists and is bounded, and its range embeds densely into $L^2(\Gamma)$, as for example $H^{1/2}_\Delta$ or some variant of it, if such an embedding of the form domain is known. However, in this abstract setting, we want to avoid spaces depending on specific operators.  
\end{bemerkung}
Next, we introduce two operators, connected to the subsidiary form $\b$. We think of them as realizations of a certain “general” operator $B$ subject to Neumann or Dirichlet boundary conditions. While this is indeed true in the setting of elliptic differential operators on domains (cf.\ $B_{\max}$ in Definition~\ref{DefBDist}), we point out that in the abstract setting considered here, it is unclear which manner to define such an operator would be the most sensible. Therefore, in this section the operator $B_0^*$ will be used as a suitable substitute for the formally undefined operator $B$, which only appears terminologically in the following definition.
\end{definition}
\begin{definition}\label{DefBDBN}
Let $\b: D(\b) \times D(\b) \rightarrow L^2(\Omega)$ be an admissible form.
\begin{enumerate}[(i)]
    \item Denote by $\lambda_\b$ the maximal semi-bound $\lambda \in \R$ such that $\Re \b(u,u) \geq \lambda \|u\|_\Omega^2,$ i.e.\ we have $\|u\|_\b^2=\Re \b(u,u)+(1-\lambda_\b) \|u\|_\Omega^2$ for $u \in D(\b)$. 
    \item The operator $B_N$ associated to $\b$ on $L^2(\Omega)$ is called the \emph{Neumann realization of $B$}.
    \item The \emph{Dirichlet realization of $B$} is the associated operator to $\b_D$, the restriction of $\b$ to $\{u \in D(\b)~|~ \tr u=0\}$.
\end{enumerate}       
\end{definition}
\begin{proposition} \label{b0gen} If $\b$ is an admissible form, then $\b_D$ is generating and symmetric, so $B_D$ is well defined and self-adjoint. Furthermore, we have 
\begin{align}\label{RepDN} D(B_N)\cap D(B_D)=\{u \in D(B_N)~|~ \tr u=0\}=D(B_N)\cap H_0^{1/2+\rho}(\Omega)\end{align}
and \[B_N u=B_D u \textrm{~for~} u\in D(B_D)\cap D(B_N).\]    
\end{proposition}
\begin{proof}
The restricted form $\b_D$ is clearly symmetric, and densely defined as the test functions still lie in $D(\b_D)$. By definition, we have $\|\cdot\|_\b=\|\cdot\|_{\b_D}$ on $D(\b_D)$, from which continuity and semi-boundedness follow. In order to show that ($D(\b_D),\|.\|_{\b_D}$) is complete as well, take a Cauchy-sequence $u_n$ with respect to $\|\cdot\|_{\b_D}=\|\cdot\|_\b$. As $\b$ is a closed form, $u_n$ converges to some $u \in D(\b)$, and by continuous embedding also in $H^{1/2+\rho}(\Omega)$. By continuity of the Dirichlet trace, the traces converge as well, whence $\tr u=0$ and $u \in D(\b_D)$ as desired. Hence it is also generating.

We verify the second part in \eqref{RepDN} first: In Lipschitz domains we have the identity $\{u \in H^s(\Omega) ~|~ \tr u =0\}=H^s_0(\Omega)$ for all $s \in (1/2,3/2)$ (cf.\ \cite[Equation~(3.7)]{BGM22}). So we directly obtain $D(B_N) \cap H_0^{1/2+\rho}(\Omega) =\{u \in D(B_N)~|~ \tr u =0\}$ due to $D(B_N) \subset D(\b) \subset H^{1/2+\rho}(\Omega)$. 

For the remaining identity assume $u \in  D(B_N)$ with $\tr u =0$. Hence there is an $f_N \in L^2(\Omega)$ such that, for all $v \in D(\b)$, $\<f_N,v\>_\Omega=\b(u,v)$ holds. But now $u \in D(\b_D)$ and in particular for all $v \in D(\b_D) \subset D(\b)$, we have \[\<f_N,v\>_\Omega=\b(u,v)=\b_D(u,v),\] which shows $u \in D(B_D)$ and $B_D u=f_N=B_N u.$ The converse is trivial. \qedhere
\end{proof}

\begin{bemerkung}~
\begin{enumerate}[(i)]
\item Note that $b_D$ itself is not admissible, as $D(\b_D)$ cannot be embedded continuously into $H^{1/2+\rho}(\Omega)$ due to the continuity of the trace, which is the reason why pure Dirichlet-Wentzell boundary conditions (without Neumann term) can not be handled via this method.
\item In the following, we are going to assume that $D(B_D)\cap D(B_N)$ is dense in $L^2(\Omega)$, which is useful to define realizations of $B$ that are in a sense minimal (or maximal) and still densely defined. We will see below, in Proposition~\ref{Prop:Dense}, that this is a very natural assumption for considering the primary form $\a$, as it will ensure that it will be densely defined as well. The simplest way to ensure this density will be to demand that $D(B_D) \cap D(B_N)$ contains the test functions. However, this excludes the case of operators of the form $-\Div Q \nabla$ for $Q$ only in $L^\infty(\Omega,\R^{\dm \times \dm})$. As we assume more regularity on $Q$ in this article, anyway, this will be no restriction for us, though.
\end{enumerate}
\end{bemerkung}
Next, we introduce two notions of generalized weak Neumann traces, the first of which is connected to a generalization of Green's first formula, while the second is closer related to the abstract notion of the associated operator and Green's second formula. 
\begin{definition}\label{DefN-NN}
Assume that $D(B_D)\cap D(B_N)$ is dense in $L^2(\Omega)$.
Define $B_0={B_N|}_{D(B_D) \cap D(B_N)}.$
Let 
$N^{\b}: D(N^{\b})  \subseteq L^2(\Omega) \rightarrow L^2(\Gamma)$ be the linear operator defined by
\interdisplaylinepenalty=10000
\begin{align*}
  D(N^{\b})\coloneqq&\{u \in D(\b)\cap D(B_0^*)~|~ \\
  &\hphantom{spacesp}\exists g \in L^2(\Gamma)\, \forall v \in D(\b): \<B_0^*u,v\>_\Omega-\b(u,v)=\<g, \tr v\>_\Gamma\}
\end{align*}
\allowdisplaybreaks
 and $N^\b u\coloneqq g.$   
Let furthermore $\NN^\b: D(\NN)  \subseteq L^2(\Omega) \rightarrow L^2(\Gamma)$ be the linear operator defined by
\begin{align*}
    D(\NN^{\b})\coloneqq&\{u \in D(B_0^*)~|~\\
    &\hphantom{spacesp} \exists g \in L^2(\Gamma) \forall v \in D(B_N): \<B_0^*u,v\>_\Omega-\<u,B_Nv\>_\Omega=\<g, \tr v\>_\Gamma\}
\end{align*} and $\NN^\b u\coloneqq g.$  
\end{definition}
We want to point out a subtlety concerning the signs: In comparison to the usual weak Neumann trace (cf.\ \eqref{Def:WeakNeuA} below) $N^\b$ and $\NN^\b$ generalize $-\partial_\nu$, as $B_0^*$ is  a generalized version of $-\Delta$.

We begin with a very simple observation that will prove to be quite useful to show equality of different traces.
\begin{lemma}\label{Magic}
Consider two linear operators $S_1: D(S_1) \subseteq V \rightarrow W$, $S_2: D(S_2) \subseteq V \rightarrow W$ on a vector spaces $V,W$. If $S_1 \subseteq S_2$, $S_1$ is surjective, and $\ker(S_2) \subseteq D(S_1)$,
then $S_1=S_2$. 
\end{lemma}
\begin{proof}
   Let $u \in D(S_2)$. As $S_1$ is surjective there is is some $v \in D(S_1)$ with $S_2v=S_1v=S_2 u.$ So $u-v \in \ker(S_2) \subseteq D(S_1)$ whence also $u=v+(u-v) \in D(S_1)$ and $S_2u=S_1u$. This shows $S_2 \subseteq S_1$ and thus equality. 
\end{proof}
We come to our first main result, which shows that the traces $N_\b$ and $\NN_\b$ are well defined. 

\begin{satz} \label{GenOP1}
In the setting of Definition~\ref{DefN-NN} we have the following.
\begin{enumerate}[\upshape(i)]
    \item $B_0$ is a densely defined, symmetric, and closed operator. Therefore, $B_0^*$ and $B_0^{**}$ are well defined and we have $B_0 \subseteq B_0^*$ as well as $B_0^{**}=B_0.$
     \item $\tr(D(B_N))$ is dense in $L^2(\Gamma)$.
     \item $N^{\b}$ and $\NN^{\b}$ are well defined, linear operators. We have $\NN^{\b}|_{D(\b)\cap D(\NN^\b)}=N^\b$ and $\ker N^\b=\ker \NN^\b=D(B_N),$ which also shows that $N^\b$ and $\NN^\b$ are densely defined. 
\item  If $N^\b$ is surjective, we have $N^\b=\NN^\b$.
\end{enumerate} 

\end{satz}
\pbo
\begin{proof}~
\begin{proofenum}[~~\upshape(i)]
\item The density follows by assumption, the closedness follows as both $B_D$ and $B_N$ are closed by default (as $\b$ and $\b_D$ are generating) and coincide on the intersection. Furthermore, as a restriction of a self-adjoint operator, $B_0$ has to be symmetric.
\item Let $f \in L^2(\Gamma)$ and $\eps >0$. As $H^{\rho}(\Gamma)$ is dense in $L^2(\Gamma)$, we find a function $u_\Gamma \in H^{\rho}(\Gamma)$ with $\|u_\Gamma-f\|_\Gamma^2 \leq \eps$. Because $\tr\colon H^{1/2+\rho}(\Omega)\to H^{\rho}(\Gamma)$ is bounded (denote its operator norm by $M$) and surjective (cf.\ \cite[Equation~2.7]{GM08}), we find a function $ u_\Omega \in H^{1/2+\rho}(\Omega)$ with $\tr u_\Omega = u_\Gamma$.
As $\b$ is admissible, the domain of the subsidiary form $D(\b)$ is densely and continuously embedded in $H^{1/2+\rho}(\Omega)$. Hence, there is a function $\hat u_\Omega \in D(\b)$ satisfying $$\|\hat u_\Omega - u_\Omega\|^2_{H^{1/2+\rho}(\Omega)}\leq M^{-2}\epsilon.$$ As for any generating form it is known that the domain of the associated operator is a form core (cf.\ \cite[Lemma 1.25]{Ouh09}), one may further approximate and even find a function $\bar u_\Omega \in D(B_N)$ such that \[\|\bar u_\Omega - \hat u_\Omega\|^2_{H^{1/2+\rho}(\Omega)}\leq C\|\bar u_\Omega -\hat u_\Omega\|^2_\b \leq M^{-2}\epsilon. \]
Altogether, we have \begin{align*} 
\| \tr \bar u_\Omega - f\|^2_\Gamma &\leq 2\|\tr u_\Omega-f\|_\Gamma^2+2\|\tr u_\Omega - \tr \bar u_\Omega\|_\Gamma^2\\ & \leq 2\eps+2 M^2\|\bar u_\Omega-\hat u_\Omega+\hat u_\Omega- u_\Omega\|^2_{H^{1/2+\rho}(\Omega)}\leq 10 \eps \end{align*} as desired.
\item The linearity of the operators is obvious. Concerning the well-definedness, assume there were two elements $g_1,g_2 \in L^2(\Gamma)$ satisfying the defining conditions, respectively. Then we have $\<g_1,\tr v\>_{\Gamma}=\<g_2,\tr v\>_{\Gamma}$ in particular for all $v \in D(B_N)$, and hence $\<g_1-g_2,\tr v\>_\Gamma=0.$ But as $\tr(D(B_N))$ is dense in $L^2(\Gamma)$ due to (ii), this implies $g_1=g_2$. 
For $u \in D(\b)$, $v \in D(B_N)$ we have $\b(u,v)=\overline{\b(v,u)}=\overline{\<B_Nv,u\>_\Omega}=\<u,B_Nv\>_\Omega.$ This shows $N^\b \subseteq \NN^\b$. 

Concerning the restriction, we assume $u \in D(\NN^\b) \cap D(\b).$ Then, as before, there is a $g \in L^2(\Gamma)$ such that, for all $v \in D(B_N) \subseteq D(\b)$, $\<B_0^*u,v\>_\Omega-\<u,B_Nv\>_\Omega=\<g,\tr v\>_\Gamma$ holds. As $u \in D(\b)$ and $v \in D(B_N)$, this implies the $L^2(\Gamma)$-function $g$ also satisfies 
\begin{equation}\label{eq:extend}
    \<B_0^*u,v\>_\Omega-\b(u,v)=\<g,\tr v\>_\Gamma
\end{equation} for all $v \in D(B_N)$. However, $D(B_N)$ is a form core for $\b$, so for any $v \in D(\b)$ there is a sequence  $(v_n)_n \in D(B_N)$ with $v_n \rightarrow v$ with respect to $\|\cdot\|_\b$ and due to the admissibility of $\b$ also in $H^{1/2+\rho}(\Omega),$ whence the trace converges as well. Using this approximation, Formula \eqref{eq:extend} can be extended to all $v \in D(\b),$ which indeed proves $\NN^\b|_{D(\b) \cap D(\NN^\b)}=N^\b.$

Next we show $N^\b$ (and thus $\NN^\b$) is densely defined. As $B_0 \subseteq B_N$ and $B_N$ is self-adjoint, we have $B_N=B_N^* \subseteq B_0^*$ which exists due to (i). Hence for $u \in D(B_N)$ (which is a dense subset of $L^2(\Omega)$) we have $\<B_0^*u,v\>_\Omega-\b(u,v)=\<B_0^*u,v\>_\Omega-\<u,B_Nv\>_\Omega=\<B_Nu,v\>_\Omega-\<u,B_Nv\>_\Omega=0$ for all $v \in D(B_N)$. So $N^\b u=0$ for any $u \in D(B_N)$. Furthermore, if $u \in D(\NN^\b)$ and $\NN^\b u=0$, then  for all $v \in D(B_N)$ we have $\<B_0^*u,v\>_\Omega-\<u,B_Nv\>_\Omega=0$. This, however, is the definition of $u \in D(B_N^*)$ and and shows $B_N^*u=B_0^*u.$ As $B_N$ is self-adjoint, this means $u \in D(B_N).$ Hence $D(B_N) \subseteq \ker N^\b \subseteq  \ker \NN^\b \subseteq D(B_N)$, which shows equality, and in particular that both operators are densely defined.

\item This is an immediate consequence of (iii) and Lemma~\ref{Magic}.\qedhere
\end{proofenum}
\end{proof}

\subsection{The system on the product space}\label{HoSys}
Next we introduce a primary form, which will be connected to the generalized system of~\eqref{10-5}--\eqref{10-9}, i.e
 \begin{alignat}{4}
  \partial_t u_1   +B_0^*(\alpha  B_N)u_1 & =0 && \textnormal{ in }   (0,\infty)\times \Omega,\label{A10-5''}\\
  \partial_t u_2 -\beta  \NN^{\b}(\alpha  B_N) u_1 + \gamma  u_2   & = 0
   && \textnormal{ on }  (0,\infty)\times \Gamma,\label{A10-6''}\\
  \NN^{\b} u_1 & = 0  && \textnormal{ on }   (0,\infty)\times \Gamma,\label{A10-7''}\\
  u_1 |_{t=0} & = u_{1,0}  && \textnormal{ in } \Omega,\label{A10-8''}\\
  u_2 |_{t=0} & = u_{2,0} && \textnormal{ on } \Gamma.\label{A10-9''}
\end{alignat} 
Throughout, we assume the following.
\begin{hyp}\label{hyp1.1} \index{Hypotheses! On the matrix $Q$ (A1)--(A3)} \index{Operators!positive definite}
Let $\Omega \subseteq \R^\dm$ be a bounded domain with Lipschitz boundary $\Gamma$. Consider $\alpha \in L^\infty(\Omega,\R)$ and $\beta, \gamma, \delta \in L^\infty(\Gamma,\R)$ such that there exists a constant $\eta>0$ with $\alpha \geq \eta$ almost everywhere on $\Omega$ and $\beta \geq \eta$ almost everywhere on $\Gamma$. Furthermore, let $\delta \geq 0$.
\end{hyp}
\begin{definition}\label{DefForm}
 Assume Hypothesis~\ref{hyp1.1} and recall Definition~\ref{DefBDBN}.
 \begin{enumerate}[(i)]
\item Let $\Hh\coloneqq L^2(\Omega,\lambda_\dm)\times L^2(\Gamma,\beta^{-1} \mathrm{dS})$ be the Hilbert space, where $\lambda_\dm$ denotes the $\dm$-dimensional Lebesgue measure and $\mathrm{d}S$ the surface measure on $\Gamma$, endowed with the canonical inner product 
\begin{equation}\label{prodH}
\<\uu,\vv\>_{\Hh}=\<u_1,v_1\>_\Omega+\<u_2,v_2\>_{\Gamma,\beta},
\end{equation}
where $\uu=(u_1,u_2),\vv=(v_1,v_2) \in \Hh$ and $$\<u_2,v_2\>_{\Gamma,\beta}=\<\beta^{-1}u_2,v_2\>_{\Gamma}=\int_\Gamma \beta^{-1}(x)u_2(x)\cdot \overline{v_2(x)} \mathrm{d}S.$$
\item Let $D(B_D)\cap D(B_N)$ be dense in $L^2(\Omega)$. Then, we define the primary form $\a: D(\a) \times D(\a) \rightarrow \C$ as
\[\a(\uu,\vv):=\<\alpha B_N u_1, B_N v_1\>_\Omega+\<\gamma u_2,v_2\>_{\Gamma,\beta}\] for all $\uu, \vv \in D(\a)$ where
 \[D(\a):=\{\uu=(u_1,u_2) \in \Hh ~|~ u_1 \in D(B_N), u_2=\tr u_1 \}. \]
 \end{enumerate}  
\end{definition}

\begin{proposition} \label{Prop:Dense}
In the situation of Definition~\ref{DefForm}~(ii), the primary form $\a$ is densely defined.
\end{proposition}
\begin{proof}~
We may assume without loss of generality that $\beta = \one$, otherwise switch to an equivalent norm. Next we exploit the density of $D(B_N)\cap D(B_D)$ in $L^2(\Omega)$. As $\big(D(B_N)\cap D(B_D)\big) \times\{0\} \subseteq D(\a)$, we have $L^2(\Omega)\times\{0\} \subseteq
\overline{D(\a)}.$ 
In order to show $\{0\}\times L^2(\Gamma) \subseteq
\overline{D(\a)},$ we use Theorem~\ref{GenOP1}~(ii) which yields that $\tr D(B_N)$ is dense in $L^2(\Omega)$. Hence, given a function $f \in L^2(\Gamma)$ and some number $\eps>0$, there is an element $\bar u_1$ of $D(B_N)$ such that $\|\tr \bar u_1 -f\|^2_\Gamma<\eps.$  Finally, we pick a function $w \in D(B_N)\cap D(B_D)$ such that $\|\bar u_1 - w\|_\Omega^2 \leq \eps$ and put
$\uu = (\bar u_1 - w, \tr (\bar u_1 -w)) = (\bar u_1 - w, \tr \bar u_1)$. Then, by construction,
we have \begin{align*}
   \| \uu - (0, f)\|^2_{\Hh}=\|\bar u_1-w\|_\Omega^2+\|\tr \bar u_1 - f\|_\Gamma^2 \leq 2\eps. 
\end{align*} 
As $f$ was arbitrary,
$\{0\}\times L^2(\Gamma) \subseteq \overline{D(\a)}$. 
Since $\overline{D(\a)}$ is a vector space, we may combine our two results and obtain $\overline{D(\a)} =\Hh$.
\end{proof}

\begin{satz} \label{GenOP2}
Assume we are in the situation in Definition~\ref{DefForm} -- including the density from part (ii). Then for $B_0$, $B_N$, $N^{\b}$, and $\NN^{\b}$ defined as in Definition~\ref{DefN-NN}, we have the following:
\begin{enumerate}[\upshape(i)]
   \item  $\a$ is a generating, symmetric form. Hence the operator $\Aa$ associated to $\a$ on $\Hh$ is self-adjoint and $-\Aa$ generates an analytic semigroup $\Tt$ on $\Hh$.
    \item $\Aa$ is given by \[\Aa=\begin{pmatrix}
    B_0^* (\alpha B_N) & 0 \\ -\beta\NN^{\b} (\alpha B_N) & \gamma
\end{pmatrix} \] on $$D(\Aa)=\{\uu \in \Hh ~|~ u_1 \in D(B_N), \alpha B_N u_1 \in D(\NN^{\b}), u_2=\tr u_1\}.$$  
\item  In particular, for $\uu_0=(u_{1,0},u_{2,0}) \in \Hh$ the Cauchy problem
    \eqref{A10-5''}--\eqref{A10-9''}
    possesses a unique solution, which is given by $\uu(t)=\Tt(t)(u_{1,0},u_{2,0})$ for $t>0$. If $N^{\b}$ is additionally surjective, we may replace $\NN^{\b}$ by $N^{\b}$ in (ii) and \eqref{A10-5''}--\eqref{A10-9''}.
\end{enumerate}

\end{satz}
\begin{proof}~
\begin{proofenum}[~~\upshape(i)]
    \item We begin by showing that $\a$ is a generating, symmetric form. As $\b$ is admissible, we have $D(\b) \subseteq H^{1/2+\rho}(\Omega)$, hence also $D(B_N) \subseteq H^{1/2+\rho}(\Omega)$ and the condition $\tr u_1= u_2$ makes sense.  
The density of $D(\a)$ has been shown in Proposition~\ref{Prop:Dense}.
Because of $\gamma \in L^\infty(\Gamma)$ the form is semi-bounded due to
$$\a(\uu,\uu)=\<\sqrt{\alpha}B_N u_1,\sqrt{\alpha}B_N u_1\>_\Omega+\<\gamma u_2,u_2\>_{\Gamma,\beta} \geq -\|\gamma\|_\infty \|\uu\|^2_\Hh.$$
The symmetry is trivial as $\alpha, \beta, \gamma$ are real-valued.
Next we consider the induced norm $\|u\|^2_\a=\a(\uu,\uu)+(1+\|\gamma\|_\infty)\|\uu\|^2_\Hh.$ With respect to this norm the form is continuous as
$$|\a(\uu,\vv)|\leq \|\sqrt{\alpha}B_Nu_1\|_\Omega\|\sqrt{\alpha}B_Nv_1\|_\Omega +\|\gamma\|_\infty \|u_2\|_{\Gamma,\beta} \|v_2\|_{\Gamma,\beta}\leq 2 \|\uu\|_\a \|\vv\|_\a.$$
Note that by definition we have $\|\sqrt\alpha B_N u_1\|^2 \leq \|u\|_\a^2$, as well as $\|\gamma\|_\infty\|u_2\|_{\Gamma,\beta}^2 \leq \|u\|_\a^2.$
Finally, we show the closedness of the form. 
Let $(\uu_n)_n \subseteq  D(\a)$ be a $\|\cdot\|_\a$-Cauchy sequence, where $\uu_n = (u_1^n, u_2^n)$.
We have to prove that this sequence converges with respect to $\|\cdot\|_\a$. Let us first note that because $\alpha$ is bounded from below by $\eta>0$, for a certain constant $C$, we have
\[
\|u_1\|_{B}^2 \leq \frac{1}{\eta} \|\sqrt{\alpha} B_N u_1\|_\Omega^2+\|u_1\|_\Omega^2 \leq C \|\uu\|_\a^2
\]
whenever $\uu = (u_1, u_2) \in D(\a)$. It follows that $(u_1^n)_n$ is a Cauchy sequence with respect to $\|\cdot\|_{B}$. As  $B_N$ is closed, we find some $u\in D(B_N)$ such that
$u_1^n \to u$ in $L^2(\Omega)$ and $B_N u_1^n\to B_N u$ in $L^2(\Omega)$. 
Next observe that for $u\in D(B_N) \subseteq D(\b)$, by definition of the associated operator and Young's inequality, we have
\begin{align*}
\|u\|^2_{H^{1/2+\rho}(\Omega)} &\leq C \|u\|_\b^2 = C(\b(u,u)+(1-\lambda_\b)\|u\|^2_\Omega)  = C((1-\lambda_\b)\|u\|_\Omega^2 +\<B_Nu,u\>_\Omega) \\
&\leq \widetilde C(\|B_N u\|_\Omega^2 + \|u\|_\Omega^2)=\widetilde C \|u\|_{B}^2
\end{align*}

for some constant $\widetilde C\geq 1$.
Combining this with the above, we observe that $u_1^n$ is also convergent in $H^{1/2+\rho}(\Omega)$ whence, by the continuity of the trace, $u_2^n = \tr u_1^n\to \tr u$ in $L^2(\Gamma)$. Setting $\uu = (u, \tr u)$, we see that
$\uu \in D(\a)$ and $\uu_n \to \uu$ with respect to $\|\cdot\|_\a$.
This proves closedness of the form. Hence $\a$ is a generating, symmetric form with a corresponding associated self-adjoint operator $\Aa$ such that $-\Aa$ generates an analytic semigroup on $\Hh$.
\item At first we define \[\Cc=\begin{pmatrix}
    B_0^* (\alpha B_N) & 0 \\ -\beta\NN^{\b}(\alpha B_N) & \gamma
\end{pmatrix} \] on $$D(\Cc)=\{\uu \in \Hh ~|~ u_1 \in D(B_N), \alpha B_N u_1 \in D(\NN^{\b}), u_2=\tr u_1\}.$$ We want to show $\Cc=\Aa.$ We begin by showing $\Cc \subseteq \Aa.$ So let $\uu \in D(\Cc) \subseteq D(\a)$, i.e.\ $u_1 \in D(B_N)$, $\alpha B_N u_1 \in D(\NN^{\b})$ and $\tr u_1=u_2.$ Then we have for all $\vv \in D(\a)$
\begin{align*}
\a(\uu,\vv)&=\<\alpha B_Nu_1,  B_Nv_1\>_\Omega+\<\gamma u_2,v_2\>_{\Gamma,\beta}\\
&=\<B_0^*(\alpha B_N u_1),  v_1\>_\Omega-\<\NN^{\b} (\alpha B_N u_1),\tr v_1\>_\Gamma+\<\gamma u_2,v_2\>_{\Gamma,\beta}\\
&=\<B_0^* (\alpha B_N u_1),  v_1\>_\Omega+\<-\beta \NN^{\b} (\alpha B_N u_1)+\gamma u_2,v_2\>_{\Gamma,\beta}=\<\Cc \uu, \vv\>_\mathcal{H}.
\end{align*}
For the reverse direction let $\uu \in D(\Aa)$ and $\Aa \uu=\ff.$ Then $\uu \in D(\a)$ and for any $\vv \in D(\a)$ we have $\a(\uu,\vv)=\<\ff, \vv\>_\Hh$. In particular, for all $\vv \in D(B_0)\times\{0\}$ (and thus for all $v_1 \in D(B_0)$) we have $$\<f_1,v_1\>_\Omega=\<\ff,\vv\>_\Hh=\a(\uu,\vv)=\<\alpha B_Nu_1, B_Nv_1\>=\< \alpha B_Nu_1, B_0 v_1\>_\Omega.$$ This shows that $\alpha B_N u_1$ is in $D(B_0^*)$ and $f_1=B_0^* (\alpha B_N u_1)$ by definition of the adjoint. 
So for all $v_1 \in D(B_N)$ we have 
\begin{align*}\<f_2,\tr v_1\>_{\Gamma,\beta}&=\a(\uu,\vv)-\<f_1,v_1\>_\Omega\\&=\<\alpha B_N u_1,B_N v_1\>_\Omega-\<B_0^* (\alpha B_N u_1),  v_1\>_\Omega+\<\gamma u_2,\tr v_1\>_{\Gamma,\beta}\end{align*} or 
$$\<\beta^{-1}(f_2-\gamma u_2),\tr v_1\>_{\Gamma}=\<\alpha B_N u_1,B_N v_1\>_\Omega-\<B_0^* (\alpha B_N u_1),  v_1\>_\Omega$$ for all $v_1 \in D(B_N),$ which shows $\alpha B_Nu_1 \in D(\NN^{\b})$ (and $\uu \in D(\CC)$) as well as $$-\NN^{\b}(\alpha B_N u_1)=\beta^{-1}(f_2-\gamma u_2)$$ or, equivalently, $f_2=-\beta\NN^{\b}(\alpha B_N u_1)+\gamma u_2,$ which shows $\Cc \uu=\ff=\Aa \uu.$ 
\item
This follows from (i) and (ii) by standard semigroup theory. For the last part we use Theorem~\ref{GenOP1}~(iv)
 \qedhere
\end{proofenum}
\end{proof}
\begin{bemerkung}\label{OrgSys}
Theorem~\ref{GenOP2}~(iii) states that the semigroup $\Tt$ governs the system~\eqref{A10-5''}--\eqref{A10-9''}. We observe that $u_1$ also solves the corresponding non-decoupled problem with Wentzell boundary conditions
\begin{alignat}{4}
  \partial_t u   +B_0^*(\alpha B_N)u & =0 && \textnormal{ in }   (0,\infty)\times \Omega,\label{A10-1''}\\
  \tr B_0^*(\alpha  B_N)u +\beta  \NN^{\b}(\alpha  B_N) u -\gamma \tr u   & = 0
   && \textnormal{ on }  (0,\infty)\times \Gamma,\label{A10-2''}\\
  \NN^{\b}u  & = 0  && \textnormal{ on }   (0,\infty)\times \Gamma,\label{A10-3''}\\
  u |_{t=0} & = u_0   && \textnormal{ in } \Omega.\label{A10-4''}
  \end{alignat}
As the semigroup is analytic, the solution is $C^\infty$ in time so that
$(\uu(t))_{t>0} = (\Tt (t) (u_{1,0}, u_{2,0}))_{t>0}$ satisfies Equations~\eqref{A10-5''} and~\eqref{A10-6''} in a classical (in time) sense. Concerning the initial system~\eqref{A10-1''}--\eqref{A10-4''}, we immediately see that $u=u_1$ solves Equation~\eqref{A10-1''},~\eqref{A10-3''} and~\eqref{A10-4''}. 

The question remains in which way the Wentzell boundary condition~\eqref{A10-2''} is satisfied. But as $\uu \in C((0,\infty), D(\Aa)^2)$ due to the analyticity of the semigroup, naturally for all $t>0$ the functions $\uu(t,\cdot)$ are in $D(\Aa^2)$ and thus we have $$\tr B_0^*(\alpha B_N) u=\tr (\Aa \uu)_1= (\Aa \uu)_2=-\beta \NN^{\b} (\alpha B_N)u+\gamma \tr u,$$ which shows~\eqref{A10-2''}. In fact, the analyticity even yields $\uu(t,\cdot) \in D(\Aa^\infty)$ for $t>0$.
\end{bemerkung}
\begin{bemerkung}\label{r.explain}
We point out that the system~\eqref{A10-1''}--\eqref{A10-4''} has to be interpreted in such a way that $u_0$ is sufficiently smooth to have a trace, say $u_0\in H^{1/2+\rho}(\Omega)$; in this setting, the solutions of~\eqref{A10-1''}--\eqref{A10-4''} are in a one-to-one correspondence with the solutions of~\eqref{A10-5''}--\eqref{A10-9''} with $u_{1,0}= u_0|_\Omega$ and $u_{2,0}= u_0|_{\Gamma}$. In our semigroup approach, however, $u_{2,0}$ can be chosen independently of $u_{1,0}$ and, by the above, all of these solutions are (distinct!) solutions of~\eqref{A10-1''}--\eqref{A10-4''}. In a way, choosing $u_{2,0}$ different from $\tr u_{1,0}$ corresponds precisely to having some free energy on the boundary, which was a main motivation to consider Wentzell boundary conditions in the first place.
\end{bemerkung}

\section{Application to strongly elliptic operators in divergence form}\label{id}
In this section, we will specify the operator $B$ to be a strongly elliptic second-order operator in divergence form and return to the investigation of the system \eqref{10-1}--\eqref{10-4}. We consider $\delta=0$ at first and deal with the Robin case at the end of Section~\ref{id2}. 
We begin by settling the precise regularity assumptions on the matrix $Q$ and  recalling some facts concerning different realizations of co-normal traces.
\subsection{Co-normal traces} \label{Tr:weak}
\begin{hyp}\label{HypQ} Assume $Q \in C^{1,1}(\overline\Omega, \R^{\dm \times \dm})$ to be symmetric and uniformly positive definite, which means there is some open superset $\widetilde \Omega \subseteq \R^\dm$ containing $\overline\Omega$ such that $Q \in C^{1,1}(\widetilde \Omega, \R^{\dm \times \dm})$ is symmetric and satisfies for some $\kappa_Q>0$ 
\begin{equation} \label{posdefA}
    \<Q(x)\xi,\xi\>_{\C^\dm} \geq \kappa_Q |\xi|^2 ~~ (x \in \widetilde \Omega, \xi \in \C^\dm).
\end{equation}

\end{hyp}
\begin{bemerkung}
The regularity $Q \in C^{1,1}(\bar\Omega,\R^{\dm \times \dm})$ is not necessary for all the subsequent steps, part of the theory can be done using only $W^{1,\infty}$-regularity. However $C^{1,1}$ is the regularity from \cite[Chapter~11]{BGM22}, and thus used when we establish higher regularity and a precise identification of the occurring traces and the domain of our operator. For a finer distinction in regularity, we refer to \cite[Hypothesis~2.5 and Section~3.1]{Plo24}.
\end{bemerkung}

\index{Trace!strong}\index{Trace!Dirichlet} \index{Trace!Neumann}
 \begin{definition}\label{Def:Strong} Let $\Omega \subseteq \R^{\dm}$ be a Lipschitz domain with outward normal $\nu$. We consider the following notions of strong traces:
    \begin{enumerate}[\upshape(i)]
        \item For a real-valued matrix $Q \in W^{1,\infty}(\bar\Omega),\R^{d\times d}$, we denote the \textit{co-normal Neumann trace} of a function $u\in C^\infty(\overline\Omega)$ by 
        $\tau^Q_N u\coloneqq \nu \cdot \tr Q \nabla u$, where we read the operator $\tr$ component-wise.
        \item For any function $\delta \in L^\infty(\Omega)$, we will call $\tau^Q_\delta=\tau^Q_N+\delta \tr$ the \textit{(co-normal) Robin trace}. 
    \end{enumerate}
 \end{definition}
 
It is known, that the Dirichlet trace extends by continuity to a bounded linear surjective operator \[\tr\colon H^s(\Omega)\to H^{s-1/2}(\Gamma) \text{~for~all~} s\in \left(\dfrac12, \dfrac32\right)\] (cf.\ \cite[Equation~(2.7)]{GM08}). In fact, this operator is even a retraction, i.e.\ there exists a continuous right-inverse.
 Even for smooth domains, however, the continuity of $\tr\colon H^s(\Omega)\to H^{s-1/2}(\Gamma)$ \index{Trace!Dirichlet} does not hold for the endpoint case $s=\frac12$, see \cite[Theorem~1.9.5]{LM72}. However, one can include the cases $s=\frac12$ and $s=\frac32$ by replacing $H^s(\Omega)$ by $H^s_\Delta(\Omega)$. In particular, it was shown in \cite[Lemma~2.3]{GM08} that the smooth trace extends to a retraction $\tr \colon H^{3/2}_\Delta(\Omega)\to H^1(\Gamma)$. 

Next we  consider the weak definition of the (co-normal) Neumann trace.
\begin{align}\label{Def:WeakNeuA}
D(\partial_\nu^Q) \coloneqq  &\big\{ u \in H^1_{\Div Q\nabla}(\Omega) \, | \, \mbox{there exists a } g\in L^2(\Gamma) \mbox{ such that}\\
& \quad\la \Div Q \nabla u, v\ra_\Omega +\la Q \nabla u, \nabla v\ra_\Omega = \la g, \tr v\ra_\Gamma \text{ for all } v\in H^1(\Omega)\big\}, \notag
\end{align}
where we set $\partial_\nu^Q u=g$. Naturally, one wants to know whether $\partial^Q_\nu$ coincides with an extension of  $\tau_N^Q$. For $Q=\id$ one has  $\partial_\nu=\tau_N\colon H^{3/2}_\Delta(\Omega)\to L^2(\Gamma)$, see \cite[Lemma~2.4]{GM08}. For $Q \neq \id$ the properties of such a possible extension were much less clear for some time. In the recent preprint \cite{BGM22}, those issues were resolved. We recall their central result for our case (\cite[Corollary~11.28]{BGM22}) adapted to the notation we are going to use. 

\begin{lemma}\label{cono}
    Let $\Omega \subseteq \R^\dm$ be a Lipschitz domain. Let $\BB$ be a formal second-order differential operator acting on elements in $L^2(\Omega)$ in a distributional sense via 
    \[\BB u = \sum_{i,j=1}^\dm \partial_i q_{ij}(x) \partial_j u,\] where the matrix $Q=(q_{ij})$ is given as in Hypothesis~\ref{HypQ}. Let $B$ denote its $L^2(\Omega)$-realization (cf.\ Definition~\ref{DefBDist}). Then, the co-normal Neumann trace defined by $u \mapsto \nu \cdot \tr(Q \nabla u)$ for smooth functions extends uniquely to  
    \begin{equation} \label{DefGamms}
    \gamma^s_Q: H^s_B(\Omega) \rightarrow H^{s-3/2}(\Gamma)
    \end{equation} for all $s \in [\frac12, \frac32 ]$, forming a compatible family in $s$. 
       Furthermore, for all $s \in [\frac12, \frac32 ]$, we have the following:
    \begin{enumerate}[\upshape(i)]
        \item The generalized Neumann traces in~\eqref{DefGamms} are surjective. In fact, there are bounded linear operators
        \begin{equation} \label{DefGammsInv}
    \Upsilon^s_Q:  H^{s-3/2}(\Gamma) \rightarrow H^s_B(\Omega),
    \end{equation} which are also compatible with each other for different $s$, and right inverses to the Neumann trace, meaning for all $\psi \in H^{s-3/2}(\Gamma)$ we have $\gamma_Q^{s}(\Upsilon_Q^{s}\psi)=\psi$. 
    \item For any $f \in H^s_B(\Omega)$ and $h \in H^{2-s}_B(\Omega)$ the following Green's formula holds:
    \begin{align*}
        \<\tr h,\gamma_Q^s f\>_{H^{3/2-s}(\Gamma)\times (H^{3/2-s}(\Gamma))'}&-\<\gamma_Q^s h,\tr f\>_{H^{s-1/2}(\Gamma)\times (H^{1/2-s}(\Gamma))'}\\
        &\hphantom{space}=\<h,Bf\>_\Omega-\<Bh,f\>_\Omega.
    \end{align*} 
    \item $\ker(\gamma_Q^s) \subseteq H^{3/2}(\Omega)$, $\ker(\tr) \subseteq H^{3/2}(\Omega)$, and for any $u \in H^{1/2}_B(\Omega)$ with either $\gamma^s_Q u=0$ or $\tr u=0$, there is some $C>0$ such that \[\|u\|^2_{H^{3/2}(\Omega)}\leq C \|u\|_\Omega^2+\|Bu\|_\Omega^2.\]
    \end{enumerate}
    \begin{proof}
        This is \cite[Corollary~11.25 and~11.28]{BGM22}.
    \end{proof}
  \end{lemma}
  We are going to verify the compatibility to the weak formulation $\partial_\nu^Q$ in Theorem~\ref{SatzA1}, later.

\subsection{On the second-order operator}\label{SecA1}
With the trace results from the last section, we are going to be able to identify the operator $\Aa$ for the case $B=-\Div Q \nabla$, and to fully describe its domain in precise terms of Sobolev regularity. The underlying subsidiary form is given as follows.

\begin{definition}\label{DefSup} Assume Hypothesis~\ref{HypQ}. Set $D(\b)\coloneqq H^1(\Omega)$, and let $\b: D(\b) \times D(\b) \rightarrow \C$ be given by 
\begin{equation}\label{Def:Bq}
  \b(u,v)=\<Q\nabla u, \nabla v\>_\Omega  
\end{equation}
for $u,v \in D(\b)$. Also, set $\b_D(u,v)\coloneqq\b(u,v)$ for $u,v \in D(\b_D)=H_0^1(\Omega)$.
\end{definition}
\begin{lemma}\label{Lem:Sett}
   The subsidiary form $\b$ defined as above is admissible in the sense of Definition~\ref{DefSubForm}. Furthermore, we have  $C_c^\infty(\Omega) \subseteq D(B_N)\cap D(B_D)$. Hence, we are in the situation of Definition~\ref{DefN-NN} and Theorem~\ref{GenOP1}.
\end{lemma}
\begin{proof}
We choose $\rho=1/2$ in Equation \eqref{TestDense} and have $C_c^\infty(\Omega) \subseteq D(\b)=H^1(\Omega)$, whence the form $\b$ is also densely defined. It is accretive, as $Q$ is uniformly positive definite. As $Q$ is also bounded, we have $$\|u\|^2_{H^1}\leq \b(u,u)+\|u\|_\Omega ^2 \leq C \|u\|_{H^1}.$$ So $\b$ is closed and continuous, and therefore $\b$ is generating. It is symmetric, as $Q$ is symmetric and real-valued.
Hence the Neumann realization $B_N$, as the associated operator to $\b$, and the Dirichlet realization, as the associated operator to $\b_D=\b|_{H_0^1(\Omega)}$, are well defined and we have $D(B_N)\cap H_0^1(\Omega)=D(B_D)\cap D(B_N)$ (cf.\ Proposition~\ref{b0gen}). 
 As $C_c^\infty(\Omega) \subseteq H_0^1(\Omega)$ we only need to show that $C_c^\infty(\Omega)  \subseteq D(B_N).$ So let $\phi \in C_c^\infty(\Omega)$, which implies that $Q \nabla \phi \in (H^1(\Omega)^\dm)$ as $C^{1,1}(\overline{\Omega},\R^{\dm \times \dm}) \subset W^{1,\infty}(\Omega,\R^{\dm \times\dm})$. Now, we may use the following version of Green's formula taken from \cite[Corollary~4.5]{BGM22}, which holds for their case of $\eps = 1/2$, as $\Delta$ maps from $H^1(\Omega)$ to $H^{-1}(\Omega)$.
    For all $v \in D(\b)=H^1(\Omega)$ we have that  
     $$\<Q \nabla \phi, \nabla v\>_\Omega+\<\Div Q \nabla \phi,v \>_\Omega=\< \nu \cdot \tr Q \nabla \phi, \tr v\>_\Gamma=0$$
     as $\nabla \phi = 0$ close to the boundary, which shows $$\b(\phi,v)=\<-\Div Q \nabla \phi,v\>_\Omega$$ for all $v \in D(\b)$ and thus $\phi \in D(B_N)$ and $B_N \phi=-\Div Q \nabla \phi.$ 
\end{proof}

So we may define $B_0$, $B_0^*$, $N^\b$, and $\NN^\b$ as stated in Theorem~\ref{GenOP1}. Furthermore, the next Lemma shows that we are in the situation where $N^\b=\NN^\b$ holds.
\pagebreak[3]
\begin{lemma}\label{Lem:Nsur}
    The operator $N^\b$ is surjective.
\end{lemma}
\begin{proof}
    The surjectivity of $N^\b$ is a special case of \cite[Lemma~3.8]{Nit11}. There it is shown that for any $g \in L^2(\Gamma)$ there is a $u \in H^1(\Omega)$ and a $\lambda \in \R$ such that for all $v \in C^1(\overline\Omega)$ we have
        $\b(u,v)+\<\lambda u,v\>_\Omega=\<g,\tr v\>_\Gamma$. Approximation in $H^1(\Omega)$ yields this result for all $v \in H^1(\Omega).$
        Hence $u \in D(B_0^*)$ as for $v \in D(B_0)$ we have $-\<\lambda u,v\>_\Omega=\b(u,v)=\<u,B_N v\>_\Omega=\<u,B_0 v\>_\Omega$ and $B_0^*u=-\lambda u,$ so $u \in D(N^\b)$ and $-N^\b u=g$. 
\end{proof}

Next we show that the operator $B_0^*$ is in some sense the maximal $L^2$-realization of $-\Div Q \nabla$, and $-N^{\b}$ coincides with the usual co-normal derivative $\partial_\nu$, which extends $\nu \cdot \tr Q \nabla.$ So we begin by verifying that both operators act as desired on smooth functions.
\begin{proposition}\label{IncMin1}
    If $\phi \in C^\infty(\overline\Omega)$, we have $B_0^* \phi=-\Div Q \nabla \phi$ and $N^{\b}\phi= -\nu \cdot \tr Q \nabla \phi$. 
\end{proposition}
\begin{proof}
    For any function $\phi \in C^\infty(\overline\Omega)$, so as seen above $Q \nabla \phi$ is an element of $(H^1(\Omega))^\dm.$ So again by \cite[Corollary~4.5]{BGM22}, for all $v \in H^1(\Omega)$, we have 
    $$\<Q \nabla \phi, \nabla v\>_\Omega+\<\Div Q \nabla \phi,v \>_\Omega=\< \nu \cdot \tr Q \nabla \phi, \tr v\>_\Gamma.$$ 
    This means, for all functions $v \in D(B_0)\subseteq D(B_N)$, that  
    $$\<-\Div Q \phi, v\>_\Omega=\<Q \nabla \phi, \nabla v\>_\Omega=\b(\phi,v)=\<\phi,B_0 v\>_\Omega,$$
     which shows $\phi \in D(B_0^*)$ and $B_0^*\phi=-\Div Q \nabla \phi.$ Furthermore, we even have $\phi \in D(N^{\b})$ and $N^{\b}\phi=-\nu \cdot \tr Q \nabla \phi.$
      \end{proof}
 We point out that $\Div Q \nabla$ does not map test functions onto test functions. Hence, there is no distributional realization of that operator and the largest space we can work on is $H^{-2}(\Omega)$, the dual of $H_0^2(\Omega)$, whence we use the $L^2(\Omega)$-realization of that version. To that end, we establish an elliptic regularity result on $\R^{dm}$ on the level of test functions. 
 \begin{lemma}\label{Lem:EllReg} Let $Q$ satisfy Hypothesis~\ref{HypQ}. Then there is a symmetric, uniformly positive definite extension $\widehat Q \in \BUC^1(\R^\dm, \R^{\dm\times \dm})$ of $Q$. Furthermore, for all $s \in [0,2]$, there is a $\lambda_0>1$ such that for any $\lambda \geq \lambda_0$ there exists $C_\lambda>0$ for which
\[\|\phi\|_{H^{s}(\R^\dm)} \leq C_{\lambda} \|(\lambda-\Div \widehat Q \nabla) \phi\|_{H^{s-2}(\R^\dm)}\] holds for all $\phi \in C_c^\infty(\Omega)$. 
\end{lemma} 
\begin{proof}
We first construct the extension. As $Q$ is uniformly positive definite in some open superset $\widetilde \Omega$ which contains $\overline\Omega$, there is a $C^\infty$-domain $\Omega'$ with $\overline\Omega \subseteq \Omega' \subseteq \widetilde\Omega,$ which can be constructed by approximation with mollified functions in the supremum norm. As $\Omega'$ is smooth with smooth boundary $\Gamma'$, there is a small tubular neighborhood  $$\Gamma'_\eps=\{x \in \R^\dm ~|~ \dist(x,\Gamma')<\eps\}$$ of $\Gamma'$, which can be parameterized by the normal vector, i.e.\ there is a smooth bijective map \[\gamma: (-\eps,\eps)\times \Gamma' \rightarrow \Gamma'_\eps; (h,x') \mapsto \gamma(h,x'):=x'+h \cdot \partial_\nu(x').\] Choosing $\eps>0$ small enough, it is possible to guarantee $\Gamma'_\eps \in \widetilde\Omega \setminus \overline\Omega,$ so $Q$ is defined on $\overline{\Omega' \cup \Gamma'_\eps}$.
For the extension, let $\psi \in C^\infty(\R,[0,1])$ be a strictly decreasing function satisfying $\psi = 1$ on $(-\infty, \frac{-\eps}{2})$ and $\psi = 0$ on $(\frac{\eps}{2}, \infty)$. Then, we obtain $\phi \in C^\infty(\R^{\dm})$ by setting $\phi = 1$ on $\Omega' \setminus \Gamma'_\eps$, $\phi = 0$ on $\R^\dm \setminus (\Gamma'_\eps \cup \Omega')$, and $\phi(x)\coloneqq\psi(h)$ for $x=\gamma(h,x') \in \Gamma'_\eps$. Hence, we can define $\widehat Q(x)=(1-\phi(x))Q(x^*)+\phi(x)Q(x)$ for an arbitrary but fixed $x^* \in \Omega$, and the new matrix $\widehat Q \in \BUC^{1}(\R^\dm, \R^{\dm \times \dm})$ still satisfies~\eqref{posdefA} as the set of uniform positive definite matrices is convex. 

For the ellipticity estimate, we use parabolic theory from \cite{DPRS22}. The constructed extension $\widehat Q$ satisfies their assumption $(S2)$ (its entries are $\BUC^1(\R^\dm)$-functions which are constant for large $|x|$), and a simple calculation shows that due to the uniform positive definiteness, the resulting operator $\lambda-\Div Q \nabla$ is also parameter-elliptic. Hence the assumptions of  \cite[Lemma~3.14]{DPRS22} are satisfied for $\sigma=0$ and $r=\lceil|s-1|\rceil=1$ and, we obtain the estimate in the $\lambda$-dependent spaces by \cite[Lemma~3.14]{DPRS22}. As the constant is allowed to depend on $\lambda$, this finishes the proof as we switch to $\lambda$-independent spaces.
\end{proof}
\begin{bemerkung}
A thorough comparison of regularities will show that a strict application of \cite[Lemma~3.14]{DPRS22} would need $q_{ij} \in \BUC^{2}(\Omega)$. However, as our operator is in divergence form, the coefficients do not need to be multipliers in $H^{s-2}(\R^\dm)$ but only in $H^{s-1}(\R^\dm)$, whence one can deduce that $\BUC^1(\Omega)$ is actually sufficient as $|s-1|\leq 1$ in our case. For details (cf.\ \cite[Section~7.2.1]{Plo24}).   
\end{bemerkung}
Now we can show that the minimal realization of $-\Div \nabla Q$ is well defined and its domain is given by $H_0^2(\Omega)$.
\begin{proposition}\label{IncMin2}
Let $B_{\min}$, the minimal realization of $-\Div Q \nabla$, be the closure of \[{B_N|}_{C_c^\infty(\Omega)}={B_0|}_{C_c^\infty(\Omega)}={(-\Div \nabla Q)|}_{C_c^\infty(\Omega)}.\] Then the following holds.
\begin{enumerate}[(i)]
\item $B_{\min}$ is well defined. 
\item On the space $C_c^\infty(\Omega)$ the graph norm $\|u\|_\Omega+\|B_0^* u\|_\Omega$ is equivalent to the full $H^2(\Omega)$-norm, whence $D(B_{\min})=H_0^2(\Omega)$.
\item  We have ${(-\Div Q \nabla)|}_{H_0^2(\Omega)}={B_{\min}}={B_0|{_{H_0^2(\Omega)}}}.$
 \end{enumerate}
\end{proposition}
\begin{proof}~
\begin{proofenum}[~~\upshape(i)]
\item The operators ${B_N|}_{C_c^\infty(\Omega)}$ and ${B_0|}_{C_c^\infty(\Omega)}$  are closable due to the self-adjointness of $B_N$. Thus $B_{\min}$ is well defined. 
\item  Let $\phi \in C_c^\infty(\Omega)$ be an arbitrary test function. Then $B_0^* \phi=-\Div Q \nabla \phi$ by Proposition~\ref{IncMin1} and thus $\|\phi\|_\Omega^2+\|B_0^*\phi\|^2_\Omega \leq \|\phi\|_{H^2(\Omega)}.$
For the reverse inequality we use the matrix $\widehat Q$ from Lemma~\ref{Lem:EllReg}. To that end, we extend $\phi$ by zero to the whole space and write $e^0 \phi$ for this extension. Choose any fixed $\lambda>\lambda_0$ where $\lambda_0$ is taken from Lemma~\ref{Lem:EllReg}. Then we have 
    $(\lambda -\Div \widehat Q \nabla) e^0 \phi = -e^0 (\Div Q \nabla \phi)+e^0(\lambda\phi),$ as $\supp (\Div Q \nabla \phi)\subseteq \supp\,  \phi \subseteq \Omega.$ Now Lemma~\ref{Lem:EllReg} with $s=2$ yields there is some $C_\lambda>0$ such that 
    \begin{align*} 
    \|\phi\|_{H^2(\Omega)}=\| e^0 \phi\|_{H^2(\R^\dm)}&\leq C_\lambda \big(\| e^0(\Div Q \nabla \phi)\|_{\R^\dm}+\lambda \|e^0 \phi\|_{\R^\dm}\big)\\ & \leq \lambda C_\lambda (\| B_0^* \phi \|_\Omega+\|\phi\|_\Omega),
    \end{align*} with $C_\lambda$ independent of the choice of $\phi$. Thus we have 
    
     \[H_0^2(\Omega)=\overline{C_c^\infty(\Omega)}^{\|\cdot\|_{H^2(\Omega)}}=\overline{C_c^\infty(\Omega)}^{\|\cdot\|_{D(B_0^*)}}=D(B_{\min}).\] 
     \item As $B_{\min}$ is closed, the first equality follows straight-forward. Now let $u \in H_0^2(\Omega)=D(B_{\min}) \subseteq D(\b_D)$. Then there is a sequence of test functions such that $\phi_n \rightarrow u$ with respect to the $H^2(\Omega)$-norm.
    For $\phi_n$ we have for all $v \in H^1(\Omega)$ 
    $$\<Q \nabla \phi_n, \nabla v\>_\Omega=\<-\Div Q \nabla \phi_n,v \>_\Omega.$$
    Due to $H^2$-convergence this also holds for $u$, and $\Div Q \nabla \phi_n$ converges to $\Div Q \nabla u \in L^2(\Omega).$ So by definition $u \in D(B_N) \cap D(B_D)=D(B_0)$ and $B_0 u=B_N u=-\Div Q \nabla u=B_{\min} u$, which shows $B_{\min} \subset B_0$. \qedhere
\end{proofenum}

   \end{proof}
Now we use duality to define $B_{\max}$. Recall that for a function $u \in L^2(\Omega)$, the induced regular distribution $[u]$ acts on a function $\phi$ via $[u](\phi)=\<u,\overline\phi\>_\Omega$.
\begin{definition}\label{DefBDist} Let $\BB: L^2(\Omega) \rightarrow (H_0^2(\Omega))'$ be defined by $$\BB u(\phi)\coloneqq\<u,-\Div Q\nabla \phi\>$$ for all $\phi \in H_0^2(\Omega)$, and define $B_{\max}$ as its $L^2(\Omega)$-realization, i.e.\ we let $$D(B_{\max})\coloneqq\{u \in L^2(\Omega) ~|~ \BB u \in (L^2(\Omega))'\},$$ and identify $B_{\max} u$ with $g \in L^2(\Omega)$, where $g$ is the unique element for which $(\BB u)(\overline\phi)=[g](\phi)$ holds.
\end{definition}
 It can also be verified that this definition is compatible with the representations $\BB u=-\sum_{i,j}\partial_i q_{ij} \partial_j u=-\Div Q \nabla u$, where each derivative is considered as weak derivative, and the multiplication with the coefficients in $H^{-1}(\Omega)$ (cf.\ \cite[Lemma~3.17]{Plo24}).   
Now we can characterize $B_{\max}$ by duality as follows:
\begin{proposition}\label{IncMax}
    We have $B_{\max}=({B_0|}_{H_0^2(\Omega)})^*=B_{\min}^*$, as well as $B_{\max}^*=B_{\min}.$ In particular, this shows that $B_{\max}$ is closed.
\end{proposition}
\begin{proof}
    Let $u \in D(B_{\max}).$ This, equivalently, means $u \in L^2(\Omega)$ and there is an  $f \in L^2(\Omega) $ such that $\<f,\phi\>_\Omega=\BB u(\phi)=\<u,-\Div Q \nabla \phi \>_\Omega=\<u,B_0 \phi\>_\Omega$ for all $\phi \in H_0^2(\Omega).$ By definition this means $u \in D(({B_0|}_{H_0^2(\Omega)})^*)$ and $f=({B_0|}_{H_0^2(\Omega)})^* u.$ The second assertion follows directly as $B_{\min}$ is closed and the restriction of the self-adjoint operator $B_N$, and therefore symmetric.
   \end{proof}  
In a final step we remove the restriction to $H_0^2$ by showing that the smooth functions are actually a core of $B_{\max}$.   

To that end, we restate the definition of the spaces $H^s_B(\Omega)$ from (cf.\ Section~\ref{Tr:weak}) in a more precise manner by setting $H^s_B(\Omega):=\{u \in H^s(\Omega)~|~ \BB u \in L^2(\Omega) \}$ for $s \geq 0$ equipped with the norm $\|u\|_{H^s(\Omega)}+\|B_{\max} u\|_\Omega$, and in particular we have $D(B_{\max})=H^0_B(\Omega)$. We will also write $\|\cdot\|_B$ instead of $\| \cdot\|_{D(B_{\max})}$ or $\| \cdot\|_{H^0_B(\Omega)}$. It might be more accurate to call those spaces $H^s_{B_{\max}}(\Omega)$, but we refrain from doing so for sake of readability, also emphasizing the fact that $B_{\max}$ actually takes the role of the abstract operator $B$ that remained undefined in Section~\ref{Abstract}.

Furthermore, we even may explicitly characterize $D(N^{\b})$ and $N^{\b}$ as we are in the setting of~\cite[Chapter~11.4]{BGM22}, whence we have Lemma~\ref{cono} at our disposal and existence,  continuity, and surjectivity of $\gamma_Q^s: H^s_B(\Omega) \rightarrow H^{s-3/2}(\Gamma)$ is assured for all $s \in [\frac12,\frac32]$. All those Neumann traces are continuous extensions of $\nu \cdot \tr Q \nabla$ on $C^\infty(\overline\Omega)$ to  $H^{s}_B(\Omega)$ by the density which we also ascertain below. 
\begin{satz}~\label{SatzA1}
    \begin{enumerate}[\upshape(i)]
        \item For all $s \geq 0$, the space $C^\infty(\overline\Omega)$ is dense in $H^s_B(\Omega)$.
        \item We have $B_0=B_{\min}$ \upshape{(}i.e. $D(B_0)=H_0^2(\Omega)$ and $C_c^\infty(\Omega)$ is a core of $B_0$\upshape{)} and $B_0^*=B_{\max}.$ 
        \item We have $N^{\b}=-\partial_\nu^Q=-\gamma^{3/2}_Q$, so in particular $D(N^{\b}) = D(\partial_\nu^Q)= H^{3/2}_B(\Omega).$
        \item $D(B_D)\subseteq H^{3/2}_B(\Omega), D(B_N)\subseteq H^{3/2}_B(\Omega).$ 
    \end{enumerate}
    \end{satz}
   \begin{proof}~
    \begin{proofenum}[~~\upshape(i)]
    \item In the case $s\geq 2$ the space $H^s_B(\Omega)$ coincides with $H^s(\Omega)$. The cases $s \in [0,2)$ will follow by adapting the proof of \cite[Lemma~2.13]{BGM22}, where this density was shown for $B=\Delta$. Consider $\dot{H}^s(\Omega)\coloneqq\{u \in H^s(\R^\dm) ~|~ \supp\, u \in \overline{\Omega} \}$ (cf.\ \cite[Section~2.3]{BGM22}) and the map $$\iota: H^{s}_B(\Omega) \rightarrow H^{s}(\Omega)\times L^2(\Omega),~~u \mapsto \iota(u)=(u,B_{\max} u),$$ which is an isometric isomorphism from $H^{s}_B(\Omega)$ to the closed subspace $\iota(H^{s}_B(\Omega))$ as $B_{\max}$ is a closed operator. Let $\Lambda$ be any functional in $(H^s_B(\Omega))'$, then $\Lambda \circ \iota^{-1}$ is a linear, bounded functional on $\iota(H^{s}_B(\Omega)),$ which can be extended to a functional $\widehat\Lambda \in (H^{s}(\Omega)\times L^2(\Omega))'=\dot H^{-s}(\Omega)\times L^2(\Omega)$ using Hahn-Banach's theorem. Hence, by \cite[p.27-29]{BGM22}, there are representatives $h_1 \in \dot H^{-s}(\Omega)$, $h_2 \in L^2(\Omega)$ such that, given any $u \in H^{s}_B(\Omega)$, for any $F \in H^{s}(\R^\dm)$, $G \in L^2(\R^\dm)$ satisfying $F|_{\Omega}=u$ and $G|_{\Omega}=B_{\max}u$ we have $$\Lambda(u)=\<h_1,F\>_{H^{-s}(\R^\dm)\times H^{s}(\R^\dm)}+\<e^0 h_2,G\>_{\R^\dm}.$$
    Note that for $s=0$, we may replace $\dot H^0(\Omega)$ with the zero extension of $L^2(\Omega)$-functions and the dual pairing with the standard inner product on $L^2(\Omega)$. In particular, if we take $u=\phi|_{\Omega}$, for any $\phi \in C_c^\infty(\R^\dm)$, we obtain  $$\Lambda(\phi|_{\Omega})=\<h_1,\phi\>_{H^{-s}(\R^\dm)\times H^{s}(\R^\dm)}+\<e^0h_2,-\Div \widehat Q \nabla \phi\>_{\R^\dm}$$ due to $(-\Div \widehat Q \nabla \phi)|_{\Omega}=(-\Div Q \nabla \phi)|_{\Omega}=B_0^* \phi|_{\Omega}=B_{\max}\phi|_{\Omega}$ by Proposition~\ref{IncMin1}, where $\widehat Q$ once more denotes the extended matrix from Lemma~\ref{Lem:EllReg}.
    
    In order to show the desired density, we assume that for any $\phi \in C^\infty(\overline\Omega)$ we had $\Lambda(\phi)=0,$ and deduce that this implies $\Lambda = 0$. By definition, however, $\Lambda(\phi)=0$ means that we have
    $$\<h_1,\phi\>_{H^{-s}(\R^\dm)\times H^{s}(\R^\dm)}=\<e^0 h_2,\Div \widehat Q \nabla \phi\>_{\R^\dm} =(\Div \widehat Q \nabla e^0h_2)(\phi)$$ for all $\phi \in C_c^\infty(\R^\dm)$ and by density for all $\phi \in H^2(\R^\dm)$. At first we consider $-\Div \widehat Q \nabla e^0 h_2=(-\Div \circ m_{\widehat Q} \circ \nabla)(e_0 h_2)$ as an element of $H^{-2}(\R^\dm)$ consisting of the separate mappings $\nabla:~L^2(\R^\dm) \rightarrow (H^{-1}(\R^\dm))^\dm$,   $m_{\widehat Q}:~ (H^{-1}(\R^\dm))^\dm \rightarrow (H^{-1}(\R^\dm))^\dm$, and  $\Div:~(H^{-1}(\R^\dm))^\dm \rightarrow H^{-2}(\R^\dm)$ where $m_{\widehat Q}$ denotes the multiplication with $\widehat Q$ in $H^{-1}(\R^{\dm})$. With the usual identification of dual spaces, we obtain 
    \begin{align*}
        \<\Div \widehat Q \nabla e^0h_2,\phi\>_{H^{-2}(\R^\dm)\times H^2(\R^\dm)}
        &=(\Div \widehat Q \nabla e^0h_2)(\phi)=\<h_1,\phi\>_{H^{-s}(\R^\dm)\times H^{s}(\R^\dm)}.
    \end{align*} So $h_1=\Div \widehat Q \nabla e^0h_2 \in H^{-2}(\R^\dm)$, or for some large $\lambda \geq \lambda_0$ also $-h_1+\lambda e^0h_2=(\lambda-\Div \widehat Q \nabla) e^0h_2.$
    Now by Lemma~\ref{Lem:EllReg} applied with $2-s \in (0,2]$, 
    we have $$\|e^0h_2\|_{H^{2-s}(\R^\dm)} \leq  C_\lambda \|\lambda e_0h_2-h_1\|_{H^{-s}(\R^{\dm})},$$
    which shows that $e^0h_2 \in H^{2-s}(\R^\dm)$ and as $\supp\,e^0h_2 \in \overline\Omega$, also $e^0h_2 \in \dot{H}^{2-s}(\Omega).$     
        
     However the space of zero extensions of $C_c^\infty(\Omega)$ lies dense in $\dot{H}^{2-s}(\Omega)$ (cf.\ \cite[(2.82)]{BGM22}), and there is a sequence of functions $(\psi_n)_n \subseteq C_c^\infty(\Omega)$ such that $e^0\psi_n$ converges to $e^0 h_2$ in $H^{2-s}(\R^\dm)$, which shows that $\Div \widehat Q \nabla e^0 \psi_n$ converges to $\Div \widehat Q \nabla e^0 h_2=h_1$ in $H^{-s}(\R^\dm)$. But then, for any $u \in H^{s}_B(\Omega)$ and $F \in H^{s}(\R^\dm)$ such that $F|_{\Omega}=u$, we obtain
    \begin{align*}
        \Lambda(u)&=\<h_1,F\>_{H^{-s}(\R^\dm)\times H^{s}(\R^\dm)}+\<e^0h_2, e^0 B_{\max} u\>_{\R^\dm}\\
        &=\lim\limits_{n \rightarrow\infty} \<\Div Q \nabla e^0 \psi_n,F\>_{H^{-s}(\R^\dm)\times H^{s}(\R^\dm)}+\<\psi_n,B_{\max} u\>_\Omega \\
        &=\lim\limits_{n \rightarrow\infty} \ -\<B_{\min} \psi_n,u\>_\Omega+\<\psi_n,B_{\min}^* u\>_\Omega=0.
    \end{align*}
 Hence, $\Lambda$ already vanishes on $H^s_B(\Omega)$; and we have shown that $\Lambda|_{C^\infty(\overline\Omega)}=0$ implies $\Lambda = 0,$ which yields the desired density by a standard corollary to Hahn--Banach. 
\item Let $v \in D(B_0)$ and $u \in D(B_{\max})$ be arbitrary. Because of (i) there is a sequence of functions $\phi_n$ in $C^\infty(\overline\Omega)$ such that $B_0^* \phi_n \rightarrow B_{\max} u$  and $\phi_n \rightarrow u$ in $L^2(\Omega).$ Hence for all $v \in D(B_0)$ we have
    \begin{align*} 0&=\<-N^{\b} \phi_n, \tr v\>_\Gamma\\
    &=\<-\NN^\b \phi_n, \tr v\>_\Gamma=\<v,B_0^* \phi_n\>_\Omega-\<B_0v,\phi_n\>_\Omega \rightarrow \<v,B_{\max} u\>_\Omega-\<B_0 v,u\>_\Omega, \,\end{align*}
    so $$\<v,B_{\max} u\>_\Omega-\<B_0 v,u\>_\Omega=0$$ for all $u \in D(B_{\max})$. This shows $v \in D(B_{\max}^*)$ and $B_0 v=B_{\max}^*v=B_{\min} v,$ which together with Propositions~\ref{IncMin2}~(iii) and \ref{IncMax} shows that $B_0=B_{\min}$ and $B_0^*=B_{\max}$ as claimed.
     \item $N^\b=-\partial_\nu^Q$ follows from $B_0^*=B_{\max}$ and the definition of the weak co-normal~\eqref{Def:WeakNeuA}. We show $\gamma_Q^{3/2}\subseteq -N^{\b} \subseteq \gamma_Q^1$ first. So let 
    $u \in H^{3/2}_B(\Omega).$ By (i) there is a sequence $(\phi_n)_{n} \subseteq C^\infty(\overline\Omega)$ that converges to $u$ in $H^{3/2}_B(\Omega).$ As in the proof of Proposition~\ref{IncMin1} we have
     $\<B_{\max} \phi_n,v\>_\Omega-\b(\phi_n,v)=\< -\nu \cdot \tr Q \nabla \phi_n, \tr v\>_\Gamma=\la-\gamma_Q^{3/2}\phi_n, \tr v\ra_\Gamma$ for all $v \in H^1(\Omega).$
     Hence, as the sequence $(\phi_n)_n$ in particular converges in $H^1_B(\Omega)$, we may take the limit and obtain for all $v \in H^1(\Omega)$
     \begin{align*} \<B_{\max} u,v\>_\Omega-\b(u,v)&=\lim\limits_{n \rightarrow \infty} \<B_{\max} \phi_n,v\>_\Omega-\b(\phi_n,v)\\
     &=\lim\limits_{n \rightarrow \infty}\<-\gamma_Q^{3/2}\phi_n, \tr v\>_\Gamma=\<-\gamma_Q^{3/2} u, \tr v\>_\Gamma.  \end{align*}
          As $\gamma_Q^{3/2} u \in L^2(\Gamma)$, by definition we have $u \in D(N^{\b})$ and $N^{\b}u=-\gamma_Q^{3/2}u.$ 
    For the second inclusion let $u \in D(N^{\b})$, then for all $v \in H^1(\Omega)$ we have
    \[\<B_{\max} u,v\>_\Omega-\b(u,v)=\<N^{\b}u, \tr v\>_\Gamma=\<N^{\b}u, \tr v\>_{H^{-{1/2}}(\Gamma) \times H^{1/2}(\Gamma)}\]
    as $v \in H^1(\Omega)$ and the Dirichlet trace maps continuously from $H^1(\Omega)$ to $H^{1/2}(\Gamma)$ (see Definition~\ref{Def:Strong}).
    
    Next, recall that $D(N^{\b})\subseteq H^1_B(\Omega)$ by definition. Using the density of $C^\infty(\overline\Omega)$ in $H^1_B(\Omega)$, we find  a sequence $(\phi_n)_n \subseteq C^\infty(\overline\Omega)$ that converges in $H^1_B(\Omega)$ towards $u$. 
    So continuity of $\gamma^1_Q$ from $H^{1}(\Omega)$ to $H^{-1/2}(\Gamma)$ yields
    \begin{align*}
     \<B_{\max} u,v\>_\Omega-\b(u,v)&=\lim\limits_{n \rightarrow \infty} \<B_{\max} \phi_n,v\>_\Omega-\b(\phi_n,v)=\lim\limits_{n \rightarrow \infty}\<-\gamma_Q^{1}\phi_n, \tr v\>_\Gamma\\
     &=\lim\limits_{n \rightarrow \infty}\<-\gamma_Q^{1}\phi_n, \tr v\>_{H^{-{1/2}}(\Gamma) \times H^{1/2}(\Gamma)}\\
     &=    \<-\gamma_Q^{1} u, \tr v\>_{H^{-{1/2}}(\Gamma) \times H^{1/2}(\Gamma)} . 
\end{align*} 
Hence we have \[ \<N^{\b}u, \tr v\>_{H^{-{1/2}}(\Gamma) \times H^{1/2}(\Gamma)}=\<-\gamma_Q^{1} u, \tr v\>_{H^{-{1/2}}(\Gamma) \times H^{1/2}(\Gamma)}\] for all $v \in H^1(\Omega).$ As the Dirichlet trace is surjective onto $H^{1/2}(\Gamma),$ we obtain 
\[ \<N^{\b}u, \psi\>_{H^{-{1/2}}(\Gamma) \times H^{1/2}(\Gamma)}=\<-\gamma_Q^{1} u,\psi \>_{H^{-{1/2}}(\Gamma) \times H^{1/2}(\Gamma)}\] for all $\psi \in H^{1/2}(\Gamma)$ by taking any solution of $\tr v=\psi$. Thus $-\gamma_Q^1 u=N^{\b}u$ on $H^{-1/2}(\Gamma),$ which in particular yields $\ker \gamma_Q^{3/2} \subseteq \ker N^{\b} \subseteq \ker \gamma_Q^1.$ By Lemma~\ref{cono} we have $\ker \gamma_Q^s \subseteq H^{3/2}_B=D(\gamma_Q^{3/2})$ for all $s \in [\frac12,\frac32]$. As $\gamma_Q^{3/2}$ is also surjective onto $L^2(\Gamma),$ we have $N^{\b}=\gamma_Q^{3/2}$ by Proposition~\ref{Magic}. 
\item This follows from Lemma~\ref{cono}, once more. We have that $u \in H^1_B(\Omega)$ and either $\tr u=0$ or $\gamma^{s}_Q u=0$ for any $s \in [\frac12,\frac32]$  implies $u \in H^{3/2}_B(\Omega).$\qedhere
\end{proofenum}        
    \end{proof} \qedhere 

\subsection{On the fourth-order system}\label{id2}
Applying the above results to the primary form $\a$ and its associated operator we obtain the following solution theorems for fourth-order system with Wentzell boundary conditions. We are able to solve the following system, as $N^\b$ is surjective by Lemma~\ref{Lem:Nsur}.

\begin{satz}\label{Cor:A1}
    Assume Hypotheses~\ref{hyp1.1} and \ref{HypQ}. Let $B=-\Div \nabla Q$. Then for $\uu_0=(u_{1,0},u_{2,0}) \in \Hh$ the Cauchy problem
    \begin{alignat}{4}
  \partial_t u_1   +B (\alpha  B) u_1 & =0 && \textnormal{ in }   (0,\infty)\times \Omega,\label{A10-5'}\\
  \partial_t u_2 +\beta  \partial_\nu^Q(\alpha  B) u_1 + \gamma  u_2   & = 0
   && \textnormal{ on }  (0,\infty)\times \Gamma,\label{A10-6'}\\
  \partial_\nu^Q u_1 & = 0  && \textnormal{ on }   (0,\infty)\times \Gamma,\label{A10-7'}\\
  u_1 |_{t=0} & = u_{1,0}  && \textnormal{ in } \Omega,\label{A10-8'}\\
  u_2 |_{t=0} & = u_{2,0} && \textnormal{ on } \Gamma\label{A10-9'}
\end{alignat}     
    possesses a unique solution, which is given by $\uu(t)=\Tt(t)(u_{1,0},u_{2,0})$ for $t>0$ where $\Tt(t)$ is the analytic semigroup generated by $-\Aa$.
    Furthermore we have \[D(\Aa)=\{\uu \in \Hh ~|~ u_1 \in H^{3/2}_B(\Omega), \alpha \Div Q \nabla u_1 \in H^{3/2}_B(\Omega), \tr u_1=u_2, \gamma^{3/2}_Q u_1=0\}.\]
\end{satz}
\begin{proof}
This is a direct consequence of Theorems~\ref{GenOP1} and~\ref{GenOP2}, whose assumptions are validated by Lemmata~\ref{Lem:Sett} and \ref{Lem:Nsur}. The identification of the operators and characterization of the domain follow from Theorem~\ref{SatzA1}.
\end{proof}
\begin{bemerkung}\label{OrgSys2}
As in Remark~\ref{OrgSys}, we observe that $u_1$ also solves the corresponding non-decoupled problem with Wentzell boundary conditions
\begin{alignat}{4}
  \partial_t u   +B (\alpha B)u & =0 && \textnormal{ in }   (0,\infty)\times \Omega,\label{A10-1'}\\
 \tr B (\alpha B \nabla)u -\beta  \partial_\nu^Q(\alpha  B) u -\gamma \tr u   & = 0
   && \textnormal{ on }  (0,\infty)\times \Gamma,\label{A10-2'}\\
  \partial_\nu^Q u  & = 0  && \textnormal{ on }   (0,\infty)\times \Gamma,\label{A10-3'}\\
  u |_{t=0} & = u_0   && \textnormal{ in } \Omega.\label{A10-4'}
  \end{alignat}
\end{bemerkung}
 \index{Trace!Robin}
 Finally, we also want to add the case $\delta>0$. The main idea is to compare the machinery of the form $\b$ with that of $\b_\delta$ defined by \[\b_\delta(u,v)=\<Q \nabla u, \nabla v\>_\Omega+\<\delta u, v\>_\Gamma\] for $0\leq \delta \in L^\infty(\Omega)$ on $D(\b_\delta)=H^1(\Omega)$.
\begin{proposition}
   Under Hypotheses~\ref{hyp1.1} and \ref{HypQ}, the subsidiary form $\b_\delta$ is admissible. We denote its associated operator by $B_{N,\delta}$.
\end{proposition}
\begin{proof}
As $0 \leq \delta \in L^\infty(\Gamma)$, the calculations are similar to the proof of Lemma~\ref{Lem:Sett}. Note that the norm $\|\cdot\|_{\b_\delta}$ is also equivalent to the full $H^1(\Omega)$-norm as the trace is continuous from $H^1(\Omega)$ to $L^2(\Gamma).$
    \end{proof}
Under the given smoothness conditions we have $C_c^\infty(\Omega)\subseteq D(B_{N,\delta})$ similar as in Lemma~\ref{Lem:Sett} due to the fact that $\b=\b_\delta$ for test functions. Naturally, one can consider the restriction of the form to $H_0^1(\Omega)$ once more, but there the form also coincides with the previous form $\b_D$, so the Dirichlet realization is independent of $\delta$. This also shows that also $D(B_{N,\delta})\cap D(B_{D,\delta})$ is dense in $L^2(\Omega)$ and Theorem~\ref{GenOP1} is applicable to $\b_\delta$, as well. The versions of all appearing operators $N^\b, \NN^\b, B_0,$ etc., associated to $\b_\delta$ will be denoted by $N^{\b_\delta}, \NN^{\b_\delta}, B_{0,\delta},$ etc.

\begin{satz}\label{Thm:Robin}
In the above setting we have the following results.
\begin{enumerate}[\upshape(i)]
    \item $N^{\b_\delta}$ is surjective.
    \item $B_{0,\delta}=B_0$,  and $(B_{0,\delta})^*=B_{\max,\delta}=B_{\max}=B_0^*=-\Div Q \nabla$ considered as $L^2$-realization of a map from $L^2(\Omega)$ to $H^{-2}(\Omega)$.
    \item $N^{\b_\delta}=-\gamma^{3/2}_Q-\delta \cdot \tr$,
     and the operator associated to $\a_\delta$ on $\Hh$ is given by  
 \begin{equation}  \label{Adelta}\Aa_\delta=\begin{pmatrix}
    \Div Q \nabla (\alpha \Div Q \nabla) & 0 \\ -\beta (\gamma^{3/2}_{Q}+\delta \tr)(\alpha \Div Q \nabla) & \gamma
\end{pmatrix} 
\end{equation} on \begin{align*} D(\Aa_\delta)=\{\uu \in \Hh ~|~ u_1 \in H^{3/2}_B(\Omega), &\alpha \Div Q \nabla u_1 \in H^{3/2}_B(\Omega),\\ & \tr u_1=u_2, 
 \gamma^{3/2}_Q u_1+\delta \tr u_1=0\}.\end{align*}
\end{enumerate}
\end{satz}
\begin{proof}~
\begin{proofenum}[~~\upshape(i)]
    \item The surjectivity of $N^{\b_\delta}$ follows from \cite[Lemma~3.7/3.8]{Nit11}, just as in the Neumann case, because the theory there also contains the Robin case (cf.\ \cite[Equation~(2.3)]{Nit11}). So $N^{\b_\delta}=\NN^{\b_\delta}$.
    \item For $u \in H_0^1$ and $f \in L^2(\Gamma)$ we have 
     $\<f,v\>_\Gamma=\<Q \nabla u,\nabla v\>_\Omega$ for all $v \in H^1(\Omega)$ if and only if we have $\<f,v\>_\Gamma=\<Q \nabla u,\nabla v\>_\Omega+\<\delta \tr u, \tr v\>_\Gamma$ for all $v \in H^1(\Omega)$ due to $\tr u=0$. Thus $D(B_N)\cap H_0^1=D(B_{N,\delta}) \cap H_0^1$ and the operators $B_N$ and $B_{N,\delta}$ coincide there, which shows $B_0=B_{0,\delta}$. Taking adjoints, this carries over to $B_0^*$. As $B_{0,\delta}=B_0=(-\Div Q \nabla)|_{H_0^2}$ we also have $\BB u=\BB_\delta u$, so also their $L^2$-realizations $B_{\max}$ and $B_{\max,\delta}$ must coincide.
     \item Due to the surjectivity shown in (i) we have $\NN^{\b_\delta}=N^{\b_\delta}$ by Theorem~\ref{GenOP1}~(iv). Next we show $N^{\b_\delta}=N^\b-\delta \tr u.$      Assume $u \in D(N^\b)$. Hence $u \in H^1(\Omega)$ and for all $v \in H^1(\Omega)$ we have $\<N^\b u,\tr v\>_\Gamma=\<B_0^* u,v\>_\Omega-\b(u,v).$ So, equivalently, $\<N^\b-\delta \tr u,  \tr v\>_\Gamma=\<(B_{0,\delta})^* u,v\>_\Omega-\b_\delta(u,v),$ which shows $D(N^\b)=D(N^{\b_\delta})$ and $N^{\b_\delta}=N^\b-\delta \tr=-\gamma_Q^{3/2}-\delta \tr$ by Theorem~\ref{SatzA1}~(iii). Furthermore, by Theorem~\ref{GenOP2}, the associated operator $\Aa_\delta$ is given by 
     \[\Aa_\delta=\begin{pmatrix}
    (B_{0,\delta})^* (\alpha B_{N,\delta}) & 0 \\ -\beta\NN^{\b_\delta} (\alpha B_{N,\delta}) & \gamma
\end{pmatrix} \] on $$D(\Aa_\delta)=\{\uu \in \Hh ~|~ u_1 \in D(B_{N,\delta}), \alpha B_{N,\delta} u_1 \in D(\NN^{\b_\delta}), u_2=\tr u_1\}.$$ Using (i), (ii) and Theorem~\ref{SatzA1}~(iii) and (iv) yields the result. Note that $\ker N^{\b_\delta}=D(B_{N,\delta})$ because of Theorem~\ref{GenOP1}~(iii) applied to $\b_\delta.$\qedhere
\end{proofenum}
\end{proof}
\begin{bemerkung}
After having verified $H^{3/2}_B(\Omega)$-regularity for the trace $N^\b$, it can be deduced that $(L^2(\Omega),\tr, N^{\b_\delta})$ actually is a quasi-boundary triple for the operator ${B_{\max}|}_{D(N^\b)}$ in the sense of \cite{BM14}, which generalizes the results from their Section 4.2 to Lipschitz domains. A detailed proof can be found in \cite[Section~3.4]{Plo24}.  
\end{bemerkung}  

We may finally collect the main result for our original system~\eqref{10-5}--\eqref{10-9}. The notations $N^\b$ and $\gamma_Q^{3/2}$ were useful in context with the general theory. In the following, however, we will write $\partial_\nu^Q,$ again, which is closer to classical notation. Note that due to our results we have $\partial_\nu^Q=\gamma_Q^{3/2}=-N^\b=-\NN^\b$, anyway.
\begin{satz}\label{Cor:A3}
    Assume Hypotheses~\ref{hyp1.1} and \ref{HypQ}. Write $B=-\Div Q \nabla$ and $\partial_\nu^Q$ as the unique extension of $\nu \cdot \tr Q \nabla$ to $H^{3/2}(\Omega)$. Then for $\uu_0=(u_{1,0},u_{2,0}) \in \Hh$ the Cauchy problem 
   \begin{alignat}{4}
  \partial_t u_1   +B(\alpha  B)u_1 & =0 && \textnormal{ in }   (0,\infty)\times \Omega,\label{3A10-5}\\
  \partial_t u_2 +\beta  \partial^Q_\nu (\alpha  B) u_1 +  \beta\delta \tr (\alpha B) u_1 + \gamma  u_2   & = 0
   && \textnormal{ on }  (0,\infty)\times \Gamma,\label{3A10-6}\\
  \partial_\nu^Q u_1+\delta u_2  & = 0  && \textnormal{ on }   (0,\infty)\times \Gamma,\label{3A10-7}\\
  u_1 |_{t=0} & = u_{1,0}  && \textnormal{ in } \Omega,\label{3A10-8}\\
  u_2 |_{t=0} & = u_{2,0} && \textnormal{ on } \Gamma\label{3A10-9}
\end{alignat}
     possesses a unique solution, which is given by $\uu(t)=\Tt(t)(u_{1,0},u_{2,0})$ where $\Tt(t)$ is the analytic semigroup generated by $-\Aa_\delta$ given as in~\eqref{Adelta}. For $t>0$, we have $u(t) \in D(\Aa_\delta^\infty)$, whence $u_1$ solves the original system with Wentzell boundary conditions given by
     \begin{alignat}{4}
  \partial_t u   +B(\alpha B)u & =0 && \textnormal{ in }   (0,\infty)\times \Omega,\label{3A10-1'}\\
 \tr B(\alpha  B)u -\beta  \partial_\nu^Q (\alpha  B) u -  \beta\delta \tr (\alpha B) u -\gamma \tr  u   & = 0
   && \textnormal{ on }  (0,\infty)\times \Gamma,\label{3A10-2'}\\
  \partial_\nu^Q u+ \delta \tr u   & = 0  && \textnormal{ on }   (0,\infty)\times \Gamma,\label{3A10-3'}\\
  u |_{t=0} & = u_0   && \textnormal{ in } \Omega.\label{3A10-4'}
\end{alignat}
\end{satz}

\section{Further properties of the solution in the fourth-order case} \label{Prop}
In this section, we present results concerning regularity and long-time behavior of our solution. We begin with a regularity result in a smoother situation. Then we return to our situation with Lipschitz domains and rougher coefficients.

\subsection{Higher regularity for smoother cases} \label{Smooth3.5}

Even with smooth coefficients and boundary, we cannot expect that for $\uu\in D(\mathcal A)$ the first component $u_1$ belongs to $H^4(\Omega)$. However, using the theory of \cite{DPRS22}, we can deduce $u_1\in H^{7/2}(\Omega).$   
Recall that any solution of the system~\eqref{10-5}--\eqref{10-9} satisfies $\uu \in D(\Aa)$, which shows $B (\alpha B)u_1 \in L^2(\Omega)$, $\beta \partial_\nu^Q (\alpha B)u_1 \in L^2$, and $\partial_\nu^Q u_1+\delta \tr u_1 =0.$ Thus the first component of any solution of~\eqref{10-5}--\eqref{10-9} in particular satisfies

\begin{alignat}{4}
  \lambda u_1  +B(\alpha  B)u_1 & = f\coloneqq \lambda u_1  +B(\alpha  B)u_1 && \textnormal{ in }   (0,\infty)\times \Omega,\label{11-5'}\\
    -\beta \partial_\nu^Q (\alpha  B) u_1  & = g\coloneqq - \beta \partial_\nu^Q (\alpha  B) u_1
   && \textnormal{ on }  (0,\infty)\times \Gamma,\label{11-6'}\\
  \partial_\nu^Q u_1 +\delta \tr u_1  & = 0  && \textnormal{ on }   (0,\infty)\times \Gamma\label{11-7'}
 \end{alignat} with $(f,g) \in \Hh=L^2(\Omega) \times L^2(\Gamma).$
We prove that the assumptions of \cite[Corollary~4.10]{DPRS22} are satisfied for this system, where $\tau=m_j+1/p$ and thus $\lceil|r'|\rceil=\lfloor|k_1'|\rfloor+1=1,\lfloor|k_2'|\rfloor+1=3.$ To that end, we assume $Q \in \BUC^4(\Omega,\R^{\dm \times \dm})$, $\alpha \in \BUC^3(\Omega)$, $\beta \in \BUC^1(\Gamma)$, and $\delta \in \BUC^3(\Gamma)$ to ensure $a_\alpha \in \BUC^1(\Omega)$, $b_{1\beta} \in \BUC^1(\Gamma)$, and $b_{2\beta} \in \BUC^3(\Gamma)$, as needed. Furthermore, we assume $\Omega$ to have $C^6$-boundary.
The last assumption to check is that the system is parameter-elliptic, for which we briefly recall the definition. Therein, we use the standard form $\binom{\lambda+A}{B}$ for parameter-elliptic boundary value problems for a moment, which is not to be confused with our operators $\Aa$ and $B$.

\begin{definition}\label{BVP}
Let $\Lambda\subseteq \C$ be a closed sector in the complex plane with vertex at the origin. Using the standard convention  $D:=-i\nabla$ for parabolic boundary value problems, let $A$ and $B=(B_1,...,B_m)$ be formally given by
\[A(x,D):=\sum_{|\alpha|\leq 2m}a_\alpha(x)D^\alpha \text{~~and~~} B_j(x,D):=\sum_{|\beta|\leq m_j}b_\beta(x) \tr D^\beta ~~(j=1,\ldots,m), \] where $m_j<2m$ \begin{proofenum}[~~\upshape(i)]
    \item We define the \textit{principal symbols} of $A$ and $B_j$ as \[a_0(x,\xi) \coloneqq \sum_{|\alpha|=2m} a_\alpha(x)  \xi^\alpha \text{~~and~~} b_{0,j}(x,\xi) \coloneqq \sum_{|\beta|=m_j} b_{j\beta}(x)\xi^\beta ~~(j=1,\ldots,m),\] respectively.    
    \item We call the family $\lambda-A(x,D)$ \textit{parameter-elliptic in $\Lambda$} if the principal symbol $a_0(x,\xi)$
satisfies
\begin{equation}\label{2-5}
 |\lambda-a_0(x,\xi) | \ge C \big( |\lambda|+ |\xi|^{2m}\big)
\quad (x\in\overline \Omega,\, \lambda\in \Lambda,\, \xi\in\R^\dm,\, (\xi,\lambda)\not=0)
\end{equation}
for some constant $C>0$.
\item The boundary value problem  $\binom{\lambda+A}{B}$ is called \textit{parameter-elliptic in $\Lambda$}
if $\lambda-A(x,D)$ is parameter-elliptic in $\Lambda$, and the following \emph{Shapiro--Lopatinskii condition} holds: \index{Shapiro--Lopatinskii}

Let $x_0\in\partial\Omega$ be an arbitrary point of the boundary; rewrite the boundary value problem $(\lambda-a_0(x_0,D),$ $b_{0,1}(x_0,D),\dots, $ $b_{0,m}(x_0,D))$ in the coordinate system associated with $x_0$ obtained from the original one by a rotation after which the positive $x_\dm$-axis has the direction of the interior normal to $\partial \Omega$ at~$x_0$. Then, for all $\xi'\in\R^{\dm-1}$ and $\lambda\in\Lambda$ with $(\xi',\lambda)\not=0$, the trivial solution $w=0$ is the only stable solution of the ordinary differential equation on the half-line
\begin{align*}
   ( \lambda -  a_0(x_0,\xi', D_\dm ))  w(x_\dm) & = 0 \quad (x_\dm\in (0,\infty)),\\
  b_{0,j}(x_0,\xi',D_\dm) w(0) & = 0 \quad (j=1,\dots,m).
\end{align*}
 
\end{proofenum}
\end{definition}
 
 \begin{lemma}\label{Lem:ParEll}
 The system  $\trinom{\lambda+B \alpha B}{-\beta \partial_\nu^Q (\alpha B)}{\partial_\nu^Q+\delta \tr}$ is parameter-elliptic in $\overline{\Sigma_\theta}$ for $\theta \in (0,\pi).$
 \end{lemma}
    \begin{proof}
     The parameter-ellipticity for the family $\lambda+B\alpha B$ is simple as we have 
    
    \begin{align*}
     a_0(x,\xi)\coloneqq{\symb}_0[B\alpha B](x,\xi)&=
      \sum_{j,k}(\xi_j q_{jk}(x) \xi_k)\alpha(x)\sum_{j',k'}(\xi_{j'} q_{j'k'}(x) \xi_{k'})>0
     \end{align*} 
     for all $|\xi|\neq 0$ as the matrix $Q=(q_{jk})_{jk}$ is symmetric and uniformly positive definite. Note for the first line that all the terms where the derivative hits the coefficient are of lower order. Hence, we can exploit homogeneity and boundedness of the domain, which shows parameter-ellipticity in any closed sector that does not contain the negative real line.
     
     A similar calculation shows
     \begin{align*}
      a'(x,\xi)\coloneqq{\symb}_0[B](x,\xi)&=\sum_{j',k'}\xi_{j'} q_{j'k'}(x) \xi_{k'},\\
     b_{0,1}(x,\xi)\coloneqq{\symb}_0[-\beta\partial_\nu^Q(\alpha B)](x,\xi)
     &=i\beta(x)\sum_{k}q_{\dm k}\xi_k \alpha(x)a'(x,\xi),\\
     b_{0,2}(x,\xi)\coloneqq{\symb}_0[\partial_\nu^Q](x,\xi)&=-i\sum_{k}q_{\dm k}(x)\xi_k.   
     \end{align*}
     In order to verify the Shapiro--Lopatinskii condition, we need to show that the ODE below only has the trivial solution. For this, the operators are locally transformed into the half-space at every fixed point $x_0 \in \Gamma$. It is a known fact that this coordinate transformation leaves the coefficients of the main symbol invariant (cf.\ \cite[Satz~10.3]{Wloka87}). Furthermore, we would like to note that, as we only consider a fixed $x_0$, the coefficients commute with all derivatives and we can simply pass from divergence to non-divergence form, which shows we may investigate the system~\eqref{SLcond} below. We write $\Xi=\binom{\xi'}{-i\partial_\dm},$ where $\xi=(\xi',\xi_\dm) \in \R^\dm$ as usual, and interpret the first components as multiplication, so $\Xi w=(\xi_1 w,\ldots,\xi_{\dm-1}w,-i \partial_\dm w)$. Then, for $x_\dm \in (0, \infty)$, and $\lambda \in \overline{\Sigma_\theta}$, $\xi' \in \R^{\dm-1}$ satisfying $(\lambda,\xi')\neq 0$, we assume
     \begin{equation}\label{SLcond}
     \begin{aligned}
         (\lambda+a_0(x_0,\Xi))w(x_\dm)&=0,\\
         b_{0,1}(x_0,\Xi)w(0)&=0,\\
         b_{0,2}(x_0,\Xi)w(0)&=0.
     \end{aligned}
     \end{equation}
     Note that integration by parts yields
     \begin{align*} \<\sum_j \Xi_j u,v\>_{L^2(0,\infty)}
     &=\<u,\sum_j \Xi_j v\>_{L^2(0,\infty)}-iu_\dm(0)\cdot \overline{ v_\dm(0)}. 
     \end{align*}
     Multiplying the first line with $w$ in $L^2((0, \infty))$ and using integration by parts and \eqref{SLcond}, we obtain
     \begin{align*}
         0&= |\lambda| \|w\|_{L^2(0,\infty)}^2+\sum_{j,k}\<\Xi_j q_{jk}(x_0)\Xi_k \alpha(x_0) a'(x_0,\Xi)w,w\>_{L^2(0,\infty)}\\
         &=|\lambda| \|w\|_{L^2(0,\infty)}^2+\sum_{j.k}\< q_{jk}(x_0)\Xi_k \alpha(x_0) a'(x_0,\Xi)w,\Xi_j w\>_{L^2(0,\infty)}\\
         &\hphantom{~~~~~~~~~~~~~~~~~~~~~~~}-\beta^{-1}(x_0)b_{0,1}(x_0,\Xi)w(0) \cdot \overline{w(0)}\\
         &=|\lambda| \|w\|_{L^2(0,\infty)}^2+\sum_{j,k}\< \Xi_k \alpha(x_0) a'(x_0,\Xi)w,q_{jk}(x_0)\Xi_j w\>_{L^2(0,\infty)}\\
         &=|\lambda| \|w\|_{L^2(0,\infty)}^2+\< \alpha(x_0) a'(x_0,\Xi)w,\sum_{k,j}\Xi_j q_{kj}(x_0)\Xi_k w\>_{L^2(0,\infty)}\\
         &\hphantom{~~~~~~~~~~~~~~~~~~~~~~~}-\alpha(x_0)a'(x_0,\Xi)w(0)\cdot \overline{b_{0,2}(x_0,\Xi)w(0)}\\
         &=|\lambda| \|w\|_{L^2(0,\infty)}^2+\|\sqrt{\alpha(x_0)}a'(x_0,\Xi)w\|_{L^2(0,\infty)}^2.
     \end{align*} 
     \normalsize Note that we used that the matrix $Q=(q_{jk})$ is symmetric and real-valued and that $\alpha>0, \beta>0$.
     If $\lambda \neq 0$, this implies that $w=0$ as desired, since $\overline{\Sigma_\theta} \subseteq \C \setminus(-\infty,0).$
     
     If $\lambda=0$, and thus by assumption $\xi' \neq 0$, we have 
     $a'(x_0,\Xi)w=0.$

     Multiplying with $w$ in $L^2(0,\infty)$ once more, we obtain
     \begin{align*}
     0&=\sum_{{j'},{k'}}\<\Xi_{j'} q_{{j'}{k'}}(x_0) \Xi_{k'} w,w\>_{L^2(0,\infty)}\\
     &=\sum_{{j'},{k'}}\<q_{{j'}{k'}}(x_0) \Xi_{k'} w, \Xi_{j'} w\>_{L^2(0,\infty)}+b_{0,2}(x_0,\Xi) w(0)\cdot \overline{w(0)}\\&=\sum_{{j'},{k'}}\<q_{{j'}{k'}}(x_0) \Xi_{k'} w, \Xi_{j'} w\>_{L^2(0,\infty)}\\
     &= \int_0^\infty \<Q (\Xi w)(x_\dm),(\Xi w)(x_\dm)\>_{\C^\dm} \mathrm{d}x_\dm     \geq \kappa_Q \int_0^\infty |(\Xi w)(x_\dm)|^2 \mathrm{d}x_\dm\\
     &=\kappa_Q (|\xi'|^2\|w\|_{L^2(0,\infty)}^2+\|\partial_\dm w\|_{L^2(0,\infty)}^2).
\end{align*} 
     In the last step we used that $Q$ is uniformly positive definite (cf.\ \eqref{posdefA}). As $\xi' \neq 0$, we also have $w=0$ in this case. Hence altogether the system is parameter-elliptic in every sector smaller than $\pi$.
 \end{proof}
   This shows $H^{7/2}$-regularity of any solution: 
 \begin{korollar}\label{Smooth7_2} Let $\Omega$ be a bounded domain with $C^6$-boundary. 
 
 Let $Q \in \BUC^4(\Omega,\R^{\dm \times \dm})$, $\alpha \in \BUC^3(\Omega)$, $\beta \in \BUC^1(\Omega)$, $\delta \in \BUC^3(\Omega)$. Then, we have 
\[ D(\mathcal A)=\{\uu \in \Hh \,|\, u_1 \in H^{7/2}(\Omega), \,B(\alpha B) u_1 \in L^2(\Omega),\, \partial_\nu^Q+\delta \tr u_1=0,\,  u_2=\tr u_1 \}.\]
     
 \end{korollar}
 \begin{proof}
 By the above all assumptions of \cite[Corollary~4.10]{DPRS22} are satisfied, hence we obtain
 $$\|u_1\|_{H^{7/2}_{\lambda}(\Omega)}\leq \|f\|_{L^2(\Omega)}+\|g\|_{B^0_{22,\lambda}(\Gamma)} < \infty$$
due to $B^0_{22,\lambda}(\Gamma)=L^2_\lambda(\Gamma)=L^2(\Gamma).$
 \end{proof}
\subsection{Hölder regularity}\label{classhoel}
We show next that on Lipschitz domains and with coefficients as in Hypotheses~\ref{hyp1.1} and \ref{HypQ} the solution $\uu(t,\cdot)=(u_1(t,\cdot),u_2(t,\cdot))$ of \eqref{10-1}--\eqref{10-4} satisfies that $u_1(t, \cdot)$ is Hölder continuous for every $t>0$. This implies that $\tr u_1(t,\cdot)=u_2(t,\cdot)$ also holds in a classical sense.

 Recall that $\Tt(t)$ maps $\Hh$ into $D(\Aa_\delta^\infty)$ for any $t>0$ as it is analytic. It is not to be expected, however, that $ \uu(t,\cdot) \in D(\Aa_\delta^\infty)$ implies $u_1(t,\cdot) \in H^{3/2+\eps}(\Omega)$ for any $\eps>0$, as such a gain in differentiability does not even necessarily hold for the much simpler Neumann Laplacian due to possible non-convex corners (cf.\ \cite{Kon67}). So Hölder-continuity cannot be derived by Sobolev embedding directly in high dimensions. However, we can use a bootstrapping idea on the integrability.

To that end we use the regular spaces $L^p(\Omega)$ and $L^p(\Gamma)$ where the coefficient $\beta$ is not included, and write $\|\cdot\|_{\Omega,p}$ and $\|\cdot\|_{\Gamma,p}$ for the occurring norms, respectively. In the case $p=2$, the index is dropped. Note that the $L^p$-spaces are nested as our domain $\Omega$ is bounded. Furthermore, recall that $C^{0,\vartheta}(\Omega)$ refers to the space of $\vartheta$-H\"older continuous functions on $\Omega$, and note that every function $u\in C^{0,\vartheta}(\Omega)$ can be extended uniquely
to a (H\"older) continuous function on $\overline{\Omega}$.

As a preparation, we establish some further results concerning weak solutions of the inhomogeneous Neumann problem
\begin{equation}\label{inhom}
\begin{alignedat}{4}
(\lambda-\Div Q \nabla) u& =\tilde f&\;&\textnormal{ in } \Omega, \\
\partial^Q_\nu u& = \tilde g&& \textnormal{ on }\Gamma.
\end{alignedat}
\end{equation}
\begin{proposition}\label{Kor:SmoothR}
Let $\tilde f \in L^2(\Omega)$, $\tilde g \in L^2(\Gamma)$.
\begin{enumerate}[\upshape(i)]
\item For $\lambda>0$,~\eqref{inhom} has a unique weak solution, by which we mean a function $u\in H^1(\Omega)$ such that
\begin{equation}\label{SolWeak}
\b^\lambda(u,v)=\la Q \nabla u, \nabla v\ra_\Omega+ \< \lambda u, v\>_\Gamma = \la \tilde f, v\ra_\Omega + \la \tilde g, \tr v\ra_\Gamma
\end{equation}
for all $v\in H^1(\Omega)$.
Furthermore, $u \in H^{3/2}_B(\Omega)$, and we have the estimate 
\[\|u\|^2_{H^{3/2}_B(\Omega)}\leq C(\|\tilde f\|^2_\Omega+\|\tilde g\|^2_\Omega).\]
\item Let $\tilde f \in L^{\dm/2+\eps}(\Omega)$, $\tilde g \in L^{\dm-1+\eps}(\Gamma)$ for some $\eps>0$. Then, for any $\lambda \in \R$, any weak solution of~\eqref{inhom} satisfies $u \in C^{0,\vartheta}(\Omega)$ for some $\vartheta \in (0,1)$ and the estimate
\begin{equation}\label{hoellift}
   \|u\|_{C^{0,\vartheta}(\Omega)}\leq C \big(\|u\|_{\Omega}+\|\tilde f\|_{\Omega,\frac{\dm}{2}+\eps}+\|\tilde g\|_{\Gamma, \dm-1+\eps}\big).
\end{equation}

If $\lambda>0$, we can drop the $\|u\|_{\Omega}$-term on the right-hand side.
\item Let $\tilde f\in L^{p}(\Omega)$, $\tilde g \in L^{p}(\Gamma)$ for some $p>2$. Then, for $\lambda>0$, the unique solution $u$ of \eqref{inhom} satisfies $(u,\tr u)\in L^{\phi(p)}(\Omega)\times L^{\phi(p)}(\Gamma)$ and 
\begin{equation}\label{plift} \|u\|_{\Omega, \phi(p)} + \|\tr u\|_{\Gamma, \phi(p)} \le C_0 \Big(\|\tilde f\|_{\Omega, p} + \|\tilde g\|_{\Gamma,p}\Big) 
\end{equation}
where
\[\phi(p) \coloneqq \begin{cases}
  \frac{\dm-2}{\dm-p}\, p & \text{ if } p\in (2,\dm),\\
  \infty & \text{ if } p\in [\dm,\infty).
\end{cases}\]
\end{enumerate}
\end{proposition}
\begin{proof}~
\begin{proofenum}
\item We construct a solution candidate by collecting properties of a weak solution. At first we observe that for any such weak solution $u \in H^1(\Omega)$ we have, given any $v \in D(B_0) \subseteq D(B_N)$, 
\[\<u,B_0 v\>_\Omega=\b(u,v)=\b^\lambda(u,v)-\lambda\<u,v\>_\Omega=\<\tilde f-\lambda u,v\>_\Omega.\] Hence $u \in D(B_0^*)$ and $\tilde f=(\lambda+B_0^*)u$. By Theorem~\ref{SatzA1}~(ii), we also have $(\lambda+B_{\max})u=(\lambda+B_0^* u)=\tilde f$ by , so the first line of~\eqref{inhom} holds in $L^2(\Omega)$, where $-\Div Q \nabla$ is seen as $L^2(\Omega)$-realization of an object in $H^{-2}(\Omega)$. Furthermore, for all $v \in H^1(\Omega)$ we obtain \[\<-B_0^*u,v\>_\Omega+\b(u,v)=\<-\tilde f,v\>_\Omega+\b^\lambda(u,v)=\<\tilde g, \tr v\>_\Gamma,\] which shows $u \in D(N^{\b})$ and $-N^{\b} u=\tilde g$. However, Theorem~\ref{SatzA1}~(iii) yields $D(N^\b)=H^{3/2}_B(\Omega)$ and  $\gamma_Q^{3/2} u=\partial_\nu^Q u=-N^\b u=\tilde g.$ So when we subtract $v=\Upsilon^{3/2}_N \tilde g$ where $\Upsilon^{3/2}_N$ is the continuous right-inverse of $\gamma_Q^{3/2}$ from Lemma~\ref{cono}~(i), the difference $u-v$ solves the Neumann problem
\begin{equation}\label{inhom2}
\begin{alignedat}{4}
(\lambda-\Div Q \nabla)(u-v)& =\tilde f+\Div Q \nabla v-\lambda v &\;&\text{ in } \Omega, \\
\gamma_Q^{3/2} (u-v)& =0&& \text{ on }\Gamma.
\end{alignedat}
\end{equation}
Thus $u-v \in D(B_N)$, as well as $(\lambda-\Div Q \nabla)(u-v)=(\lambda+B_N)(u-v)$.
In conclusion, we have shown that any weak solution of~\eqref{inhom} satisfies $(\lambda+B_N)(u-v)=\tilde f+\Div Q \nabla v-\lambda v.$
Hence, a suitable solution candidate is given by
$\tilde u\coloneqq(\lambda+B_N)^{-1}(\tilde f+\Div Q \nabla \Upsilon^{3/2}_N \tilde g -\lambda \Upsilon^{3/2}_N \tilde g)+\Upsilon^{3/2}_N \tilde g$, which is well defined due to $(-\infty,0) \in \rho(B_N)$.

Finally, we verify $\tilde u$ indeed is a solution and satisfies the regularity estimate.
By Lemma~\ref{cono}~(iii), we have 
\begin{align*} 
\|\tilde u-v\|^2_{H^{3/2}_B(\Omega)}&\leq C( \|\tilde u-v\|^2_\Omega+\|B_N(\tilde u-v)\|_\Omega)=C\|\tilde u-v\|^2_{B_N} \\
&=C\|(\lambda+B_N)^{-1}(\tilde f+\Div Q \nabla v-\lambda v)\|^2_{B_N} \leq C\big(\| \tilde f\|_\Omega^2+\|v\|_{H^{3/2}_B(\Omega)}^2\big)
\end{align*}
(where the constant $C$ is generic) and thus by the above and the continuity of $\Upsilon_N^{3/2}$
\[ \|\tilde u\|^2_{H^{3/2}_B(\Omega)}\leq  C \big(\|\tilde f\|^2_\Omega+\|v\|^2_{H^{3/2}_B(\Omega)}\big)\leq C (\|\tilde f\|^2_\Omega+\|\tilde g\|^2_\Gamma). \]
Now, naturally, $\tilde u$ solves~\eqref{inhom} in a strong sense by construction. By definition of $-N^\b=\partial_\nu^Q$ it also satisfies~\eqref{SolWeak} and is in particular a weak solution with all the desired properties.

The uniqueness is straightforward: If we had two weak solutions $u,w \in H^1(\Omega)$ satisfying~\eqref{SolWeak} for all $v \in H^1(\Omega)$, we would have $\b^\lambda(u-w,v)=0$ for all $v \in H^1(\Omega).$ Now $\lambda+B_N$ is the associated operator to $\b_\lambda$, whence $u-w$ in $D(B_N)$ and $(\lambda+B_N)(u-w)=0.$ Again, by $(-\infty,0) \subseteq \rho(B_N)$, we have $u-w=0$. 
    \item This is \cite[Theorem~3.1.6]{Nit10} applied to $A(x,u,p)=Q \cdot p$, $a(x,u,p)=\lambda u$. Note that their Assumption 2.9.1 is satisfied, and we are in the situation of \cite[Remark~3.1.7]{Nit10}. If $\lambda>0$, we can estimate
    \[\|u\|_\Omega \leq \|u\|_{H^{3/2}_B(\Omega)}\leq C(\|\tilde f\|^2_\Omega+\|\tilde g\|^2_\Omega) \leq C(\|\tilde f\|^2_{\frac{\dm}{2}+\eps,\Omega}+\|\tilde g\|^2_{\dm-1+\eps,\Omega}).\] 
    \item Let $\lambda>0$. By (i) and (ii) the unique solution satisfies the two estimates 
\begin{equation}\label{8-2}
\|u\|_{\Omega, 2}+\|\tr u\|_{\Gamma, 2} \leq C \|u\|_{H^{3/2}_B(\Omega)} \leq C\big (\|\tilde f\|_{\Omega, 2}+\| \tilde g\|_{\Gamma, 2} \Big),
\end{equation}
as well as 
\begin{equation}\label{8-1}
\|u\|_{\Omega, \infty} + \|\tr u\|_{\Gamma, \infty} \leq \|u\|_{C^{0,\vartheta}(\Omega)} \leq
C\big( \|\tilde f\|_{\Omega, \dm} + \|\tilde g\|_{\Gamma, \dm}\big).
\end{equation}
More precisely, the solution operator $R_\lambda$ that maps $(\tilde f,\tilde g)$ to $(u,\tr u)$ is well defined and continuous from $X_0\coloneqq L^2(\Omega) \times L^2(\Gamma)$ to $Y_0\coloneqq L^2(\Omega) \times L^2(\Omega)$
as well as from $X_1\coloneqq L^\dm(\Omega) \times L^\dm(\Gamma)$ to $Y_1\coloneqq L^\infty(\Omega) \times L^\infty(\Omega)$. 
By complex interpolation, we obtain that $R_\lambda$ is also continuous from $[X_0,X_1]_\theta$ to $[Y_0,Y_1]_\theta$ for all
$\theta\in (0,1)$. To identify the interpolation spaces, recall from \cite[Theorem~1.18.1]{Tri95} that
complex interpolation of tuples of $L^p$-spaces yields the tuple of interpolated spaces in the sense of
\[ [L^{p_0}(\Omega)\times L^{q_0}(\Gamma), L^{p_1}(\Omega)\times L^{q_1}(\Gamma)]_\theta = [L^{p_0}(\Omega),L^{p_1}(\Omega)]_\theta \times
[L^{q_0}(\Gamma)\times L^{q_1}(\Gamma)]_\theta \]
for all $p_0,p_1,q_0,q_1\in [1,\infty]$.
Moreover, we have the equality $[L^{p_0}(\Omega), L^{p_1}(\Omega)]_\theta = L^p(\Omega)$ (and a similar equality for $\Gamma$) for $\frac 1p=\frac{1-\theta}{p_0} + \frac\theta{p_1}$ in the sense of equivalent norms, see \cite[Theorem~1.18.6/2]{Tri95}.  From this, we obtain for all
$\theta\in (0,1)$ the
continuity of $R_\lambda: X_\theta \to Y_\theta$ where
$X_\theta \coloneqq  L^{p}(\Omega) \times L^{p}(\Gamma)$
and $Y_\theta \coloneqq L^{\phi(p)}(\Omega)\times L^{\phi(p)}(\Gamma)$ with  $p$ and $\phi(p)$ being
defined by $\frac 1p = \frac{1-\theta}2+ \frac \theta \dm$ and
$\frac 1{\phi(p)}=\frac{1-\theta}2$. For $p\in (2,\dm)$, the first equation yields
$\theta = \frac{\dm(p-2)}{(\dm-2)p}$, and the second equation gives
\[ \phi(p) = \frac{2}{1-\theta} = \frac{\dm-2}{\dm-p}\; p .\]
This proves the assertion for $p\in (2,\dm)$. For $p\ge \dm$ the statement follows
directly from~\eqref{8-1}.\qedhere
\end{proofenum}
\end{proof}
However, the estimate we actually would like to make use of would be of type \eqref{plift} for solutions of the inhomogeneous \emph{Robin} problem with $\lambda=0$, i.e.
\begin{equation}\label{inhomrob}
\begin{alignedat}{4}
-\Div Q \nabla u& =f&\;&\text{ in } \Omega, \\
\partial^Q_\nu u + \delta \tr u& =g&& \text{ on }\Gamma,
\end{alignedat}
\end{equation}
because in order to obtain higher regularity for the Wentzell problem we decouple it into two underlying Robin problems of precisely that form. 
Though Hölder continuity for the Robin case is also established in \cite[Example~4.2.7]{Nit10}, this is not helpful in our situation, since an explicit estimate of type~\eqref{hoellift} is not given there due to the complexity of the bootstrapping argument, and we cannot deduce~\eqref{plift}, as before.
To avoid this obstacle, we rewrite the Robin problem into a Neumann problem, to which we apply Proposition~\ref{Kor:SmoothR}. The price we pay is that the solution $u$ appears on the right-hand side and we have to assume a priori that its integrability is as high as the data's.
\begin{lemma}\label{IPolr}
  Let $\dm\ge 2,\, p\in (2,\infty)$. Then, there is a constant $C_0>0$ such that whenever $u \in H^{3/2}_B(\Omega)$ is a weak solution of~\eqref{inhomrob} with $f, u \in L^{p}(\Omega)$, as well as $g, \tr u \in L^{p}(\Gamma)$, we have $(u,\tr u)\in L^{\phi(p)}(\Omega)\times L^{\phi(p)}(\Gamma)$ and
\[ \|u\|_{\Omega, \phi(p)} + \|\tr u\|_{\Gamma, \phi(p)} \le C_0 \Big( \|u\|_{\Omega, p} + \|f\|_{\Omega, p} + \|g\|_{\Gamma,p}+\|\tr u\|_{\Gamma,p}\Big) \]
where
\begin{equation}\label{DefPhi}
    \phi(p) \coloneqq \begin{cases}
  \frac{\dm-2}{\dm-p}\, p & \textnormal{ if } p\in (2,\dm),\\
  \infty & \textnormal{ if } p\in [\dm,\infty).
\end{cases}
\end{equation}
\end{lemma}

\begin{proof}
Let $u \in H^1(\Omega)$ be a weak solution of~\eqref{inhomrob}. Then, given any $v \in D(B_0) \subseteq D(B_N)$, 
$$\<u,B_0 v\>_\Omega=\b(u,v)=\b_\delta(u,v)=\<f,v\>_\Gamma.$$ Hence, $u \in D(B_0^*)$, $f=B_0^*u=(B_{0,\delta})^*u$ by Theorem~\ref{Thm:Robin}~(ii). Once more, we have $B_{\max}u=B_0^*u=f$ by Theorem~\ref{SatzA1}~(ii), so the first line of~\eqref{inhomrob} holds in $L^2(\Omega)$, where $-\Div Q \nabla$ is seen as $L^2(\Omega)$-realization of an object in $H^{-2}(\Omega)$. Furthermore, for all $v \in H^1(\Omega)$ we obtain \[\<-(B_{0,\delta})^* u,v\>_\Omega+\b_{\delta}(u,v)=\<g, \tr v\>_\Gamma\] for all $v \in H^1(\Omega)$, which shows $u \in D(N^{\b_\delta})$ and (by Theorem~\ref{Thm:Robin}) $-N^{\b_\delta} u=\partial_\nu^Q u+\delta \tr u=g$. Therefore, $u$ also solves~\eqref{inhom} with $\lambda=1$, $\tilde f=f+u \in L^2(\Omega)$, and $\tilde g=g-\delta \tr u \in L^2(\Gamma)$, whence it must coincide with this problem's unique solution. Hence, Proposition~\ref{Kor:SmoothR} is applicable and as, due to the extra assumption, $(u, \tr u)$ is also an element of  $L^p(\Omega) \times L^p(\Gamma)$, so is $(\tilde f,\tilde g)$. Then, by Proposition~\ref{Kor:SmoothR} we have 
\begin{align*} \|u\|_{\Omega, \phi(p)} + \|\tr u\|_{\Gamma, \phi(p)} &\le C \Big(  \|\tilde f\|_{\Omega, p} + \|\tilde g\|_{\Gamma,p}\Big) \\
&\le C_0 \Big( \|u\|_{\Omega, p} + \|f\|_{\Omega, p} + \|g\|_{\Gamma,p}+\|\tr u\|_{\Gamma,p}\Big)\end{align*} as desired.
\end{proof}

We obtain the following corollary about the integrability of elements of $D(\Aa_\delta)$, where $\phi(r)$ is defined as in~\eqref{DefPhi}.

\begin{korollar}\label{boot}
Let $r> 2$.
If $\uu \in D(\Aa_\delta) \cap (L^r(\Omega)\times L^r(\Gamma))$, $\Aa_\delta\uu \in L^r(\Omega)\times L^r(\Gamma)$ and $(\alpha \Div Q \nabla u_1, \tr \alpha \Div Q \nabla u_1) \in L^r(\Omega)\times L^r(\Gamma)$, then
$\uu \in L^{\phi(r)}(\Omega)\times L^{\phi(r)}(\Gamma)\textrm{~and~}(\alpha \Div Q \nabla u_1 , \tr \alpha \Div Q \nabla u_1) \in L^{\phi(r)}(\Omega)\times L^{\phi(r)}(\Gamma).$
\end{korollar}

\begin{proof}
By Theorem~\ref{Thm:Robin}, we have for $\uu \in D(\Aa_\delta)$
\begin{align*}
(\Aa_\delta\uu)_1&=\Div Q \nabla \alpha (\Div Q \nabla u_1),\\
 (\Aa_\delta\uu)_2&=-\beta(\partial_\nu^Q+\delta \tr)(\alpha \Div Q \nabla u_1)+\gamma u_2.
\end{align*}
Thus, if $\uu$ satisfies the assumption of this corollary, then $w=-\alpha \Div Q \nabla u_1$
solves the inhomogeneous Robin problem
\begin{align*}
-\Div Q \nabla w&=(\Aa_\delta\uu)_1 \in L^r(\Omega) \\
  (\partial_\nu^Q+\tr \delta)w&=\beta^{-1}(\Aa_\delta\uu)_2-\beta^{-1}\gamma u_2 \in L^r(\Gamma).
\end{align*}
As $(\alpha \Div Q\nabla u_1,\tr \alpha \Div Q\nabla u_1) \in L^r(\Omega)\times L^r(\Gamma)$ by assumption, so is $(w,\tr w)$. Hence, by Lemma~\ref{IPolr}, $(w, \tr w) \in L^{\phi(r)}(\Omega) \times L^{\phi(r)}(\Gamma)$, which also implies  $\Div Q \nabla u_1 \in L^{\phi(r)}(\Omega)$ as the functions $\alpha,\alpha^{-1}$ are bounded.
Naturally $L^{\phi(r)}(\Omega)\subseteq L^{r}(\Omega)$.
Since $\uu \in D(\Aa_\delta)  \cap (L^r(\Omega)\times L^r(\Gamma))$, we also know that $(\partial_\nu^Q+\delta \tr)u_1=0$ and $(u,\tr u)=(u_1,u_2) \in L^r(\Omega)\times L^r(\Gamma)$, whence $u_1$ solves the homogeneous Robin problem
\begin{align*}
-\Div Q \nabla u_1&=-\Div Q \nabla u_1 \in L^{r}(\Omega), \\
  (\partial_\nu^Q+\delta \tr) u_1&=0 \in L^{r}(\Gamma).
\end{align*}
Applying Lemma~\ref{IPolr} once more yields $u_1 \in L^{\phi(r)}(\Gamma)$ and $u_2=\tr u_1 \in L^{\phi(r)}(\Gamma)$, as claimed.
\end{proof}

We can now prove the main result of this section.

\begin{satz}\label{HoelReg}
Let $\uu \in D(\Aa_\delta^\infty)$. Then $u_1 \in C^{0,\vartheta}(\Omega)$ for some $\vartheta \in (0,1)$. In particular, this shows Hölder continuity of $u_1(t,\cdot)$ for $t>0$ due to the analyticity of the semigroup $\Tt$.
\end{satz}

\begin{proof}
Let $\uu \in D(\Aa_\delta)$. Then $u_1,\alpha \Div Q \nabla u_1 \in H^{3/2}_\Delta(\Omega) \subseteq H^1(\Omega)$. Furthermore, $\tr u_1$, $\tr (\alpha \Div Q \nabla u_1) \in H^1(\Gamma)$. By Sobolev embedding  (see \cite[Theorem~4.12]{AF03}), we obtain $H^1 \subseteq L^\frac{2\dm}{\dm-2}$. Hence, for $\dm \leq 4$, we have $H^1 \subseteq L^\dm,$ which shows 
\[D(\Aa_\delta)\subseteq\{\uu \in L^{\dm}(\Omega)\times L^{\dm}(\Gamma)~|~ \big (\alpha \Div Q \nabla u_1,\tr (\alpha \Div Q \nabla u_1)\big ) \in  L^{\dm}(\Omega) \times L^{\dm}(\Gamma)\}.\]
If $\dm\geq 5$, we have $u_1,\alpha \Div Q \nabla u_1 \in L^\frac{2\dm}{\dm-2}(\Omega)$ and $\tr u_1,\tr (\alpha \Div Q \nabla u_1) \in L^\frac{2\dm}{\dm-2}(\Gamma).$ 
This shows \[D(\Aa_\delta)\subseteq\{\uu \in L^{r_1}(\Omega)\times L^{r_1}(\Gamma)~|~ \big (\alpha \Div Q \nabla u_1,\tr (\alpha \Div Q \nabla u_1)\big ) \in  L^{r_1}(\Omega) \times L^{r_1}(\Gamma)\}\] for $r_1=\frac{2\dm}{\dm-2}.$
Inductively, we obtain
\[D(\Aa_\delta^k)\subseteq\{\uu \in L^{r_k}(\Omega)\times L^{r_k}(\Gamma)~|~ \big (\alpha \Div Q \nabla u_1,\tr (\alpha \Div Q \nabla u_1)\big ) \in  L^{r_k}(\Omega) \times L^{r_k}(\Gamma)\},\] where $r_k=\phi(r_{k-1})=\phi^{k-1}(\frac{2\dm}{\dm-2})$.
Indeed, assume this statement is true for some $k$ and consider $\uu\in D(\Aa_\delta^{k+1})$. Then $\uu \in D(\Aa_\delta^k) \subseteq D(\Aa_\delta)$
and $\Aa_\delta\uu \in D(\Aa_\delta^k)$. By induction hypothesis, $\uu, \Aa_\delta\uu \in   L^{r_{k}}(\Omega)\times  L^{r_{k}}(\Gamma)$, and $\big (\alpha \Div Q \nabla u_1,\tr (\alpha \Div Q \nabla u_1)\big ) \in L^{r_k}(\Omega) \times L^{r_k}(\Gamma)$. Hence, Corollary~\ref{boot} yields $\uu \in L^{\phi(r_{k})}(\Omega)\times  L^{\phi(r_{k})}(\Gamma)=L^{r_{k+1}}(\Omega)\times  L^{r_{k+1}}(\Gamma)$ as well as \[\big (\alpha \Div Q \nabla u_1,\tr (\alpha \Div Q \nabla u_1)\big )\in L^{\phi(r_{k})}(\Omega)\times  L^{\phi(r_{k})}(\Gamma)=L^{r_{k+1}}(\Omega)\times  L^{r_{k+1}}(\Gamma).\]
From the structure of the map $\varphi$ it is clear that $(r_k)_{k\in\mathbb \N}$ is an increasing sequence that tends to $\infty$.
Hence for all $\dm \in \N$ we have found a $k_0\in \mathbb{N}$ such that 
\[D(\Aa_\delta^{k_0}) \subseteq \{\uu \in L^\dm(\Omega)\times L^\dm(\Gamma)~|~\big (\alpha \Div Q \nabla u_1,\tr (\alpha \Div Q \nabla u_1)\big ) \in  L^{\dm}(\Omega) \times L^{\dm}(\Gamma)\} .\] For any such $\uu \in D(\Aa_\delta^{k_0})$,
we have $-\Div Q \nabla u_1=-\Div Q \nabla u_1\eqqcolon \tilde f \in L^{\dm}(\Omega)$ as well as $\partial_\nu^Q u_1=-\delta \tr u_1\eqqcolon \tilde g \in L^\dm(\Gamma)$ due to $\alpha \in L^\infty(\Omega)$, $\delta \in L^\infty(\Gamma)$. Thus $u_1$ is a weak solution of~\eqref{inhom} for $\lambda=0$, and Proposition~\ref{Kor:SmoothR}~(ii) implies $u_1 \in C^{0,\vartheta}(\Omega)$ as claimed.
\end{proof}

\begin{bemerkung}\label{CorHoel}
The proof of Theorem~\ref{HoelReg} actually yields a number $k_0\in \mathbb{N}$, depending only on the dimension $\dm$, such that $\uu \in D(\Aa_\delta^{k_0})$ implies $u_1\in C^{0,\vartheta}(\Omega)$. The number $k_0$ we calculated there is not sharp, but the embedding certainly does not hold for dimensions that are too large. 

For example, for the case of Neumann boundary conditions, i.e.\ $\delta=0$, we can verify $D(\Aa) \subseteq C^{0,\vartheta}(\Omega)$ for $\dm\leq 6$, simply as $\alpha \Div Q \nabla u_1 \in H^{3/2}(\Omega) \subseteq L^\frac{2\dm}{\dm-3}(\Omega) \subseteq L^{\dm/2+\eps}(\Omega)$ from which the assertions follows from Proposition~\ref{Kor:SmoothR}~(ii) for $\lambda=0$ (cf.\ \cite[Theorem~7.34]{AF03}). 

But, at least for constant $Q$ and $\alpha$ it is quite simple to construct functions in $D(\Aa)$ which are not Hölder continuous for $\dm \geq 8$. As Sobolev embeddings are sharp, we know that for $\dm\geq 8$ the Sobolev space $H^4(\Omega)$ is not contained in $L^\infty(\Omega)$.
Now, let $\Omega'$ be a smooth domain contained in $\Omega$.
Let $v$ be a function that lives in $H^4(\Omega')$ but not in $L^\infty(\Omega')$. As $\Omega'$ is smooth there exists an extension of $v$ (denoted by $v$ again) to $H^4(\R^\dm)$, which still cannot be in $L^\infty(\R^\dm)$.  Now let $\phi \in C_c^\infty(\Omega)$ such that $\phi = 1$ on $\Omega'$. Then the function $u=\phi \cdot v \in H^4(\Omega) \setminus L^\infty(\Omega)$. Moreover, it satisfies the Neumann boundary condition $\partial_\nu u=0$ as it is compactly supported on $\Omega$. Hence $(u,\tr u) \in D(\Aa)$ (as all the other regularity conditions are implied by $H^4$-regularity because $Q$ and $\alpha$ are constant). However, $(u,\tr u) \not\in L^\infty(\Omega) \times L^\infty(\Gamma)$, so $\uu$ cannot be Hölder continuous.
\end{bemerkung}
\subsection{Asymptotic behavior and eventual positivity}
\label{spec}
In this section, we derive asymptotic properties of our solution. We are going to skip the proofs whenever neither variable coefficients nor extra Robin-term are relevant and the ideas can be carried over directly from \cite[Chapter~6]{Denk-Kunze-Ploss21}, as is the case for the next two results.

\begin{lemma}
The operator $\Aa_\delta$  has compact resolvent.
\end{lemma}

\begin{korollar}\label{c.spectrum}
There exists an orthonormal basis $(\ee_n)_{n}$ of $\Hh$ consisting of eigenfunctions of $\Aa_\delta$, say
$\Aa_\delta \ee_n = \lambda_n \ee_n$, where the sequence $\lambda_n$ is increasing to $\infty$, allowing the representation 
\[\Aa_\delta \ff= \sum_{k=1}^\infty \lambda_k\<\ff,\ee_k\>_{\Hh}\ee_k\] for all $\ff \in D(\Aa_\delta)$. Moreover, as $\ee_n \in D(\Aa_\delta^\infty)$, it has a H\"older continuous representative in the sense that
there exists a function $e_n \in C^{0,\vartheta}(\Omega)$ such that $\ee_n = (e_n|_{\Omega}, e_n|_{\Gamma})$.

Finally, for all $\ff \in \Hh$, the semigroup $\Tt$ can be represented as
\begin{align} \label{sgdecomp}
\uu(t)=(u_1(t),u_2(t))=\Tt (t)\ff=\sum_{k=1}^\infty e^{-\lambda_k t}\<\ff,\ee_k\>_{\Hh}\ee_k.
\end{align}
\end{korollar}

\begin{lemma}\label{kernel}~

\begin{enumerate}[\upshape(i)]
\item If $\gamma= 0$ almost everywhere, then we have $\ker(\Aa_\delta)\subseteq {\rm span}(\ind{\Omega},\ind{\Gamma})$ and $\int_\Gamma \delta |\tr u_1|^2\textrm{d}S=0$.
\item If $\gamma = 0$, $\delta=0$ almost everywhere, then $\lambda_1=0$ and $\ker(\Aa)={\rm span}(\one_\Omega, \one_\Gamma).$
\item If $\gamma, \delta \geq 0$ and either $\gamma>0$ or $\delta>0$ on a set of positive surface measure, then 
$\lambda_1 >0$ and we have $\ker(\Aa_\delta)=\{0\}$.
\item If $\int_\Gamma \gamma\, \dx S <0$, then $\lambda_1<0$.
\end{enumerate}
\end{lemma}

\begin{proof}
In cases (i)--(iii), we have $\gamma, \delta \geq 0$. Hence, $\a_\delta$ is accretive, so we have $\lambda_1 \geq 0$. Thus, whether $\lambda_1=0$ or $\lambda_1>0$ depends only on $\ker (\Aa_\delta)$.\smallskip

 (i) Suppose $\gamma= 0$ almost everywhere, then $\uu \in \ker(\Aa_\delta)$ implies $u_1 \in \ker (B_{N,\delta})$. More precisely, let $\uu \in \mathrm{ker}(\Aa_\delta)\subseteq D(\Aa_\delta)\subseteq D(\a_\delta)$. Then
\[
0=\<\Aa_\delta \uu,\uu\>_{\mathscr{H}}=\a_\delta(\uu,\uu)=\int_\Omega \alpha |B_{N,\delta} u_1|^2 \mathrm{d}x.
\]
It follows that $\alpha |B_{N,\delta} u_1|^2=0$ and hence, since $\alpha (x)\geq \eta$, $u_1 \in \ker (B_{N,\delta})$. This means \[0=\<Q \nabla u_1, \nabla u_1\>_\Omega +\<\delta u_1, u_1\>_\Gamma=\|\sqrt Q \nabla u_1\|_\Omega^2+\|\sqrt\delta \tr u_1\|^2\]
and shows $\nabla u_1=0$, whence $u_1$ is constant (and, therefore, also $u_2=\tr u_1$). Moreover, $\int_\Gamma \delta |\tr u_1|^2 \mathrm{d}S=0$.

(ii) If $\delta=0$, naturally also the converse holds as $B_N \one_\Omega=-\Div Q \nabla \one_\Omega=0$ and the constant functions satisfy the Neumann boundary condition, which shows $\lambda_1=0.$ 

(iii) If  $\gamma=0$, $\delta>0$ on a set of positive surface measure $\Gamma_0$, due to (i), $u_1$ is still constant. But now, we also find a set of positive measure $\Gamma_\eps  \subseteq \Gamma$ where $\delta > \eps$.  However, if we had $u = c \neq 0$ by (i), this would yield \[0=\int_\Gamma \delta |\tr u_1|^2 \mathrm{d}S \geq \int_{\Gamma_\eps} \delta |\tr u_1|^2 \mathrm{d}S\geq \eps c^2 S(\Gamma_\eps) > 0.\] 
Analogously, if $\gamma>0$ on a set of positive measure and $ \uu \in \ker (\Aa_\delta)$, we calculate
\begin{align*}
 0= \<\Aa_\delta \uu,\uu\>_{\mathscr{H}}=\a(\uu,\uu) &=\int_\Omega \alpha |\Delta u_1|^2 \mathrm{d}x+\int_\Gamma \beta^{-1}\gamma |u_2|^2 \dx S \\
 &\geq \int_{\Gamma_\eps} \beta^{-1}\gamma c^2\, \mathrm{d}S  \geq \|\beta\|^{-1}_\infty \eps c^2 S(\Gamma_\eps)>0,   
\end{align*}

another contradiction.
\smallskip

(iv) Plugging $(\one_\Omega,\one_\Gamma)\in D(\a_\delta)$ into the usual Rayleigh quotient, we obtain a negative value as
$\int_\Gamma \beta^{-1} \gamma \, \dx S < 0$, and thus $\lambda_1 < 0$.
\end{proof}
This yields the following asymptotic behavior of the semigroup $\Tt$.

\begin{satz}\

\begin{enumerate}[\upshape(i)]
\item If $\gamma = 0$, $\delta=0$ almost everywhere, then $\|\Tt(t)\ff - \bar\ff\|_\Hh\leq e^{-\lambda_2 t}\|\ff\|_{\Hh}$ for all $\ff\in\Hh$, where
\[
\bar \ff \coloneqq \frac{1}{\lambda_{\dm}(\Omega)+\int_\Gamma \beta^{-1}\mathrm{d}S}\left(\int_\Omega f_1  \mathrm{d}x+\int_\Gamma \beta^{-1}f_2  \mathrm{d}S\right)(\one_\Omega,\one_\Gamma),
\]
and $\lambda_2>0$ is the second eigenvalue of $\Aa$.
\item If $\gamma, \delta \geq 0$ and $\gamma>0$ or $\delta > 0$ on a set of positive measure, then $\|\Tt(t)\ff \|_\Hh\leq e^{-\lambda_1 t}\|\ff\|_{\Hh}$ holds for all $\ff\in\Hh$. Thus, in this case, the semigroup $\Tt$ is exponentially stable.
\item If $\int_\Gamma \gamma \, \dx S < 0$, then $\|\Tt(t)\| = e^{-\lambda_1 t} \to \infty$ as $t\to \infty$.
\end{enumerate}
\end{satz}

\begin{proof}
For (i) observe that in this case $\lambda_1=0$ and $\bar \ff = e^{-\lambda_1 t} \la \ff, \ee_1\ra_{\Hh}\ee_1$ in view of Lemma~\ref{kernel}. Thus
\eqref{sgdecomp} and Parseval's identity yield
\[
\|\Tt(t)\ff - \bar \ff\|_\Hh^2= \sum_{k=2}^\infty e^{-\lambda_k t}|\la \ff, \ee_k\ra_{\Hh}|^2 \leq
\sum_{k=2}^\infty e^{-\lambda_2 t}|\la \ff, \ee_k\ra_{\Hh}|^2 \leq e^{-\lambda_2 t}\|f\|^2_{\Hh}.
\]
This proves (i). In case (ii), we have $\lambda_1>0$ (see again Lemma~\ref{kernel}), and (ii) follows by a similar computation.
(iii) follows by considering an eigenvalue corresponding to the eigenvalue $\lambda_1$.
\end{proof}
\begin{bemerkung}\label{IdH}
For our following positivity investigations, we consider the Hilbert lattice $H=L^2(\Omega \,\dot\cup\, \Gamma,\mu)$, where $\mu$ is given by $\mu(A)=\int_{A\cap \Omega} \one \,\mathrm{d}\lambda+\int_{A \cap \Gamma} \beta^{-1}\mathrm{d}S$, with positivity cone \begin{equation}\label{cone} H^+\coloneqq\{u \in H ~| ~ u(x)\geq 0 ~\mu\text{-a.e} \}.
\end{equation} Once can easily show that $H$ can be identified with our Hilbert space $\Hh=L^2(\Omega)\times L^2(\Gamma,\beta^{-1}\mathrm{d}S)$ from Definition~\ref{DefForm}~(i), and $u$ with $\uu=(u_1,u_2).$    
\end{bemerkung}

\begin{proposition}\label{nonpos}
Let $\gamma\geq 0$, then the semigroup $\Tt$ generated by $-\Aa$ is real, but neither positive nor $L^\infty$-contractive.
\end{proposition}

\begin{proof}
The same arguments as in \cite[Lemma~3.5]{Denk-Kunze-Ploss21} work as Robin and Neumann boundary conditions are equal for test functions.
\end{proof}
Again however, we may show that the semigroup $\Tt$ is \emph{eventually} positive in the sense of  \cite{DGK16b}, \cite{DGK16a} and \cite{DG18}, as the critical ingredient is the Hölder continuity established in the previous section also for the variable coefficient case. We recall a simplified version of the definition from \cite[Section~1]{DG18} and the used criterion.  
\begin{definition}\label{EvPos}
Let $(\Omega,\mu)$ be a $\sigma$-finite measure space and $T$ a real $C_0$-semigroup on the space $H=L^2(\Omega,\mu)$ with positivity cone $H^+$ given in the sense of~\eqref{cone}. Then, $T$ is called \emph{eventually positive} if there is some time $t_0>0$ such that for all $f \in H_+ \setminus\{0\}$ and $t \geq t_0$ there is some $\eps>0$ for which $T(t)f \geq \eps$ holds $\mu$-almost everywhere. 
\end{definition}

\begin{satz}\label{CritEvP}
    Let $(\Omega,\mu)$ be a $\sigma$-finite measure space and $T$ a real $C_0$-semigroup generated by a self-adjoint operator $A$ on $H=L^2(\Omega,\mu)$. If $D(A^\infty) \subseteq L^\infty(\Omega,\mu)$, then the following assertions are equivalent:
    \begin{enumerate}[{\upshape(i)}]
        \item $T$ is eventually positive.
        \item $\ker(s(A)-A)$ is one-dimensional and contains a vector $v$ such that $v \geq \eps$ holds $\mu$-almost everywhere for some $\eps>0$.
    \end{enumerate}  
\end{satz}

\begin{satz}\label{EventPos}
Let $\gamma=\delta=0$. Then the semigroup $\Tt$ is eventually positive in the sense of Definition~\ref{EvPos}.
\end{satz}

\begin{proof}
We apply Theorem~\ref{CritEvP} for $H$ defined as in Remark~\ref{IdH} and $A=-\Aa$. As $\beta, \beta^{-1}$ are bounded, $L^\infty(\Omega\,\dot\cup\, \Gamma,\mu)$ can be identified with $L^\infty(\Omega) \times L^\infty(\Gamma).$
$\Tt$ is real as a consequence of Proposition~\ref{nonpos}, and the operator $\Aa$ is self-adjoint due to the symmetry of the form (see Theorem~\ref{GenOP2}). Finally, we have that $D(\Aa^\infty)$ embeds into $L^\infty(\Omega)\times L^\infty(\Gamma)$ by Theorem~\ref{HoelReg}.
Now, to deduce eventual positivity, we only have to verify assertion (ii) of Theorem~\ref{CritEvP}. But this a direct consequence of Lemma~\ref{kernel}~(ii) that yields $s(-\Aa)=\lambda_1=0$ as $\gamma=0$, $\delta=0$, and $(\one_\Omega,\one_\Gamma) \in \ker(\Aa).$ \qedhere \end{proof}

For a counterexample for eventual positivity in the case $\gamma>0$, $\delta=0$, $Q=I_\d$ we refer the reader to \cite[Section~7]{Denk-Kunze-Ploss21}.

\section{Application to systems of higher order}
In this short excursion, we point out that the abstract theory established in Chapter~\ref{Abstract} is not necessarily restricted to problems of order 4, though identification of the abstract operators and regularity characterizations can be more difficult.
\begin{proposition}
 Let $\Omega \subset \R^d$ be a bounded Lipschitz domain, $H=L^2(\Omega)$, and let $\Delta_N$ denote the Neumann Laplacian. For $k \in \N$ consider the form 
$\b_k=\<\Delta^k u,\Delta^k v\>_\Omega$ on $D(\b_k)=D(\Delta_N^k)$. Then $\b_k$ is admissible and $B_N=\Delta_N^{2k},$ whence also $C_c^\infty(\Omega) \subset D(B_N)\cap D(B_D)$ is satisfied, and we are in the situation of Theorem~\ref{GenOP2} and Remark~\ref{OrgSys}.
\end{proposition}
\begin{proof}
The form $\b_k$ is symmetric and accretive by default. We have $\|u\|^2_{\b_k}=\|u\|^2_\Omega+\|\Delta^k u\|_\Omega^2.$ As $\Delta_N^k$ is a closed operator. $D(\Delta_N^k)$ is a Hilbert space with respect to the graph norm and hence the form is closed. The continuity follows from Cauchy--Schwarz's inequality. Thus $\b_k$ is generating and its associated operator $B_N$ is self-adjoint.
We have $C_c^\infty(\Omega) \subseteq D(\Delta_N^k) \subseteq D(\Delta_N) \subseteq H^1(\Omega)$ and may choose $\rho=1/2$ in \eqref{TestDense}. For the last embedding we calculate  
$\|u\|^2_{H^1(\Omega)}=\<\Delta u,u\>_\Omega\leq C( \|u\|_\Omega^2+\|\Delta u\|_\Omega^2).$ 
That $D(\Delta^k_N)$ is continuously embedded in $D(\Delta^N)$ follows by the closed graph theorem applied to the identity id: $D(\Delta^N_k) \rightarrow D(\Delta_N)$. $D(\Delta_N^k)$ is also dense in $H^1(\Omega)$ due to \cite[Proof of Lemma 1.25]{Ouh09} and the analyticity of the semigroup. We begin by showing that $B_N$ coincides with $\Delta^{2k}$ distributionally.
Let $u \in D(B_N)$ and $v \in C_c^\infty(\Omega) \subset D(\Delta_N^{2k})$, then $\<f,\phi\>=\<\Delta_N^k u, \Delta_N^k \phi\>=\<(\Delta_N^k)^* u, \Delta_N^k \phi\>=\<u,\Delta^{2k} \phi\>$ as $\Delta^k \phi \in D(\Delta_N^k).$ Hence $f=\Delta^{2k} u$ as $\Delta_N$ and thus $\Delta_N^k$ is self-adjoint. So $D(B_N) \subset \{u \in D(\Delta_N^k)~|~ \Delta_N^{2k}u \in L^2(\Omega)\}.$ On the other hand $D(\Delta_N^{2k})\subset D(B_N)$ as for the same reason for $u \in D(\Delta_N^{2k})$ the relation $\<\Delta_N^{2k} u, v\>=\<\Delta_N^k u, \Delta_N^k v\>$ holds for all $v \in D(\Delta_N^k)$. So $\Delta^{2k}_N \subset B_N =B_N^* \subset (\Delta_N^{2k})^*=\Delta_N^{2k}.$  
\end{proof}
Now we obtain $D(B_0)=D(B_D)\cap D(B_N)=D(B_N)\cap H_0^1(\Omega)$ by Proposition~\ref{b0gen}. Thus $D(B_0)=\{u \in D(\Delta_N^{2k})~|~ \tr u=0\}.$
Now for $u \in D(B_0^*)$, there is a $f \in L^2(\Omega)$ such that for all $v \in D(B_0)$ we have $\<u, B_0 v\>=\<u, \Delta^{2k}v\>=\<f,v\>_\Omega$, so $B_0^*$ is a restriction of the maximal $L^2$-realization of the distribution $\Delta^{2k}$. 

Hence (taking $\alpha=\beta=1, \gamma=\delta=0$ in Remark~\ref{OrgSys}) we can find a solution of the system 
\begin{alignat}{4}
  \partial_t u   +\Delta^{4k}u & =0 && \textnormal{ in }   (0,\infty)\times \Omega,\label{D10-1''}\\
  \tr \Delta^{4k}u +\NN^{\b}(\Delta^{2k}) u   & = 0
   && \textnormal{ on }  (0,\infty)\times \Gamma,\label{D10-2''}\\
  \NN^{\b}u  & = 0  && \textnormal{ on }   (0,\infty)\times \Gamma,\label{D10-3''}\\
  u |_{t=0} & = u_0   && \textnormal{ in } \Omega,\label{D10-4''}
  \end{alignat}

  where $N_\b$ is given by the abstract definition
\begin{align*}
    D(\NN^{\b})\coloneqq&\{u \in D(B_0^*)~|~\\
    &\hphantom{spacesp} \exists g \in L^2(\Gamma) \forall v \in D(\Delta_N^{2k}): \<\Delta^{2k}u,v\>_\Omega-\<u,\Delta^{2k} v\>_\Omega=\<g, \tr v\>_\Gamma\}
\end{align*}
and $\NN^\b=g$.
The identification of $\NN$ is more difficult here, but some progress can be made
using the theory of quasi-boundary triples (cf.\ \cite{BM14}, \cite[Chapter~8.6]{BHS20}).

For all $u \in H^0_\Delta(\Omega)$ and $v \in D(\Delta_N)$ we have
\begin{align}\label{GreenGen}
    \<\Delta u, v\>_\Omega-\<u, \Delta v\>_\Omega=\<\tilde \tau_N u, \tr v\>_{\Gg_0' \times \Gg_0},
\end{align}
where $\tilde \tau_N$ is the extension of the Neumann trace from $H^0_\Delta(\Omega)$ to the space $\Gg_0'$ introduce there, which is the dual space of $\tr D(\Delta_N)$ equipped with a Hilbert space structure.
For this version of Green's formula also see \cite[Proposition~2.4]{Denk-Kunze-Ploss21}. 
This allows us to characterize $\NN^\b$ at least on a subset of $D(B_0^*)$. 
Let $$V=\{u \in L^2(\Omega) ~|~ \Delta^{2k} u, \Delta^{2k-1} u \in L^2(\Omega), \Delta^{j}u \in \ker{\tilde \tau_N}, \textrm{~for~} j=0,...,2k-2\}.$$
\begin{lemma}
Let $u \in D(\NN^\b)\cap V$, then $\Delta^j u \in H^{3/2}_\Delta(\Omega)$ for $j=0,...,2k-1$ and $\NN^\b u=\partial_\nu \Delta^{2k-1}u.$  
\end{lemma}
\begin{proof}
As we have $\Delta^j u \in L^2(\Omega)$ for $j=0,..., 2k$ and $D(B_0) \subset D(B_N)=D(\Delta_N^{2k})$ for $u \in D(\NN_b) \cap V$ and $v \in D(\Delta_N^{2k})$ we have
\begin{align*}
\<g,\tr v\>&= \<\Delta^{2k} u, v\>_\Omega-\<u, \Delta^{2k} v\>_\Omega=\sum_{j=0}^{2k-1}\<\tilde \tau_N \Delta^j u, \tr \Delta^{2k-1-j} v\>_{\Gg_0' \times \Gg_0}\\
&=\<\tilde \tau_N \Delta^{2k-1} u, \tr v\>_{\Gg_0' \times \Gg_0}.
\end{align*}
Now as $D(\Delta_N^{2k})$ is dense in $H^1(\Omega)$, $\tr D(\Delta_N^{2k})$ is dense in $L^2(\Gamma)$. Thus we have $g=\tilde \tau_N \Delta^{2k-1}u \in L^2(\Gamma)$ and 
 \begin{align*}
 \<\Delta^{2k} u, v\>_\Omega-\<\Delta^{2k-1}u, \Delta v\>_\Omega&=\<\tilde \tau_N \Delta^{2k-1}u, \tr v\>_{\Gg_0' \times \Gg_0}=\<g,\tr v\>_\Gamma
 \end{align*}
for all $v \in D(\Delta_N^{2k})$, and by approximation also for all $v \in D(\Delta_N)$ (as $D(\Delta_N^{2k})$ is a core of $\Delta_N$). Finally, we obtain $\Delta^{2k-1} u \in H^{3/2}_\Delta(\Omega)$ and $\NN_b u=\partial_\nu \Delta^{2k-1} u$ for $u \in D(\NN_b) \cap V$ by \cite[Proposition~2.4~(ii)]{Denk-Kunze-Ploss21} and similarly $\Delta^j u \in H^{3/2}_\Delta(\Omega)$ for all $j=0,...,2k-1$.
\end{proof}

\end{document}